\documentclass[a4paper,11pt,reqno]{amsart}

\usepackage[T1]{fontenc}

\usepackage[a4paper,top=2cm,bottom=2cm,left=2.2cm,right=2.2cm]{geometry}

\usepackage{titlesec}
\usepackage{lipsum}
\usepackage{enumerate}
\usepackage{enumitem}
\usepackage{mathrsfs}
\usepackage{amsmath,amssymb,amsthm}
\usepackage{mathtools}
\usepackage{graphicx, float, subfigure}
\usepackage{url}
\usepackage{tikz}
\usepackage{bbold}
\usepackage{color}
\usepackage{xcolor}

\usepackage[all,cmtip]{xy}
\usepackage{multirow}
\usepackage{makecell}
\usepackage[colorinlistoftodos]{todonotes}
\usepackage[colorlinks=true, allcolors=blue]{hyperref}
\usepackage{harpoon}
\usepackage{MnSymbol}
\usepackage{esvect}
\usepackage{diagbox}
\usepackage[title]{appendix}

\DeclareMathOperator{\sgn}{sgn}
\DeclareMathOperator{\wt}{wt}

\theoremstyle{plain}
\newtheorem{theorem}{\scshape Theorem}[section]
\newtheorem{proposition}[theorem]{\scshape Proposition}
\newtheorem{lemma}[theorem]{\scshape Lemma}
\newtheorem{corollary}[theorem]{\scshape Corollary}
\newtheorem{conjecture}[theorem]{\scshape Conjecture}

\newtheorem*{assumption*}{\scshape Assumption}

\newtheorem*{claim*}{Claim}

\theoremstyle{definition}

\newtheorem{definition}[theorem]{\scshape Definition}
\newtheorem{remark}[theorem]{\scshape Remark}

\newcommand{\M}{\operatorname{M}}
\newcommand{\pf}{\mathsf{pf}}
\renewcommand{\det}{\mathsf{det}}
\renewcommand{\wt}{\mathsf{wt}}

\numberwithin{equation}{section}

\titleformat{\section}{\centering\bfseries}{\thesection}{1em}{\MakeUppercase}
\titleformat{\subsection}{\bfseries}{\thesubsection}{1em}{}

\begin{document}
\title{Off-diagonally symmetric domino tilings of the Aztec diamond of odd order}
\author{Yi-Lin Lee}
\address{Department of Mathematics, Indiana University, Bloomington, Indiana 47405}
\email{yillee@iu.edu}

\begin{abstract}
We study the enumeration of off-diagonally symmetric domino tilings of odd-order Aztec diamonds in two directions: (1) with one boundary defect, and (2) with maximally-many zeroes on the diagonal. In the first direction, we prove a symmetry property which states that the numbers of off-diagonally symmetric domino tilings of the Aztec diamond of order $2n-1$ are equal when the boundary defect is at the $k$th position and the $(2n-k)$th position on the boundary, respectively. This symmetry property proves a special case of a recent conjecture by Behrend, Fischer, and Koutschan.

In the second direction, a Pfaffian formula is obtained for the number of ``nearly'' off-diagonally symmetric domino tilings of odd-order Aztec diamonds, where the entries of the Pfaffian satisfy a simple recurrence relation. The numbers of domino tilings mentioned in the above two directions do not seem to have a simple product formula, but we show that these numbers satisfy simple matrix equations in which the entries of the matrix are given by Delannoy numbers. The proof of these results involves the method of non-intersecting lattice paths and a modification of Stembridge's Pfaffian formula for families of non-intersecting lattice paths. Finally, we propose conjectures concerning the log-concavity and asymptotic behavior of the number of off-diagonally symmetric domino tilings of odd-order Aztec diamonds.
\end{abstract}

\maketitle

\section{Introduction and statement of results}

In 1992, Elkies, Kuperberg, Larsen and Propp~\cite{ELKP1, ELKP2} introduced domino tilings of the Aztec diamond. The \textit{Aztec diamond} of order $n$, denoted by $AD(n)$, is the union of all unit squares inside the diamond-shaped region $\{(x,y) \in \mathbb{R}^2 : |x|+|y| \leq n+1 \}$. A \textit{domino tiling} of the Aztec diamond is a covering of $AD(n)$ using dominoes (union of two unit squares sharing an edge) without gaps or overlaps. They proved that the number of domino tilings of $AD(n)$, denoted by $\M(AD(n))$, is given by the following elegant formula.
\begin{equation}\label{eq.adn}
  \M(AD(n)) = 2^{\frac{n(n+1)}{2}}.
\end{equation}

Given a set of combinatorial objects $X$ and a finite group $G$ acting on it, a \textit{symmetry class} is a collection of $H$-invariant objects of $X$, where $H$ is a subgroup of $G$. The symmetry classes of plane partitions and alternating sign matrices have been studied extensively for the past forty years (see for instance \cite[Section 6]{Kra15S} and \cite[Section 2]{BFK23}). Two of the symmetry classes of domino tilings of $AD(n)$ (quarter-turn invariant and half-turn invariant) have been enumerated by Yang \cite{Yang91} and Ciucu \cite[Section 7]{Ciucu97} in the 1990s. In the previous paper \cite{Lee23}, the author introduced and enumerated a new symmetry class of domino tilings of $AD(n)$, which is called the \textit{off-diagonal} symmetry class.

In \cite{Lee23}, the author proved that (as shown in Theorem~\ref{thm.OffPfaffian}) the number of off-diagonally symmetric domino tilings of $AD(n)$ with some boundary defects (some unit squares are removed from the boundary of the Aztec diamond) can be expressed by the Pfaffian of a matrix $A_I$, where $A_I$ is the matrix obtained from $A$ by selecting rows and columns indexed by $I$ (depending on these boundary defects). In particular, if there are no boundary defects, then off-diagonally symmetric domino tilings only exist for Aztec diamonds of even order. This can be explained in two ways: (1) the number of off-diagonally symmetric domino tilings of an odd-order Aztec diamond is given by the Pfaffian of an odd-order matrix, which is zero; (2) using a well-known bijection (see Section~\ref{sec.methodpath}) between domino tilings of $AD(n)$ and $n$-tuples of non-intersecting Delannoy paths, off-diagonally symmetric domino tilings of $AD(n)$ are in bijection with $n$-tuples of non-intersecting paths where the endpoints of these paths are in pairs (this is due to an ``off-diagonal'' condition \cite[Lemma 7]{Lee23}), which can only occur when $n$ is even.

In this paper, we consider off-diagonally symmetric domino tilings of Aztec diamonds of odd order in the following two directions.
\begin{enumerate}
  \item Off-diagonally symmetric domino tilings of $AD(2n-1)$ with one boundary defect (Section~\ref{sec.off}). This implies that the corresponding Pfaffian is of even order, so such domino tilings exist. Our first main result (Theorem~\ref{thm.main1}) gives a symmetry property which states that the numbers of off-diagonally symmetric domino tilings of $AD(2n-1)$ are equal when the unit square at the $k$th position and the $(2n-k)$th position on the boundary are removed, respectively.
  \item Nearly off-diagonally symmetric domino tilings of $AD(2n-1)$ (Section~\ref{sec.nearoff}). From the bijection mentioned in the previous paragraph, domino tilings are in bijection with $(2n-1)$-tuples of non-intersecting paths. Among these paths, the ending points of $2n-2$ paths are paired, while the remaining one path has no such restriction. This corresponds to the condition that there are maximally-many ``zeroes'' on the diagonal (as defined in Section~\ref{sec.nearoff}). The second main result (Theorem~\ref{thm.main2}) provides a Pfaffian expression of nearly off-diagonally symmetric domino tilings of $AD(2n-1)$.
\end{enumerate}
Our data indicates that the numbers of domino tilings considered in the above two directions do not seem to have a simple product formula. Nevertheless, these numbers satisfy simple matrix equations which are presented in our third main result (Theorem \ref{thm.main3}).

\subsection{Off-diagonally symmetric domino tilings of the Aztec diamond with one boundary defect}\label{sec.off}

Given $AD(n)$, we consider the checkerboard coloring of the square lattice with the unit squares along its top right boundary colored black. A \textit{cell} is a $2 \times 2$ square with bottom left and top right unit squares colored black while top left and bottom right colored white. Given a domino tiling of $AD(n)$, we assign a number $-1, 0$ or $1$ to each cell if it contains $0, 1$ or $2$ complete domino(es), respectively. This $(-1,0,1)$-assignment of cells from a given domino tiling of $AD(n)$ corresponds to an \textit{$n \times n$ alternating sign matrix}\footnote{An \emph{alternating sign matrix of order $n$} is an $n \times n$ matrix with entries $0, 1$ or $-1$ such that all row and column sums are equal to $1$ and the non-zero entries alternate in sign in each row and column.}, which was first provided by Ciucu \cite[Section 2]{Ciucu96}. We refer the interested reader to \cite[Section 1.2]{Lee23} for more details.

The concept of off-diagonally symmetric domino tilings of the Aztec diamond was motivated by \textit{off-diagonally symmetric alternating sign matrices} introduced by Kuperberg \cite{Kup02}, where these matrices are symmetric with respect to the diagonal and the diagonal entries are all zero. A domino tiling of $AD(n)$ is said to be \textit{off-diagonally symmetric} if
  \begin{itemize}
    \item the tiling is symmetric about the vertical diagonal of $AD(n)$, and
    \item the $n$ cells along the vertical diagonal are assigned $0$. In other words, each such cell contains exactly one complete domino.
  \end{itemize}

In addition, we consider some boundary defects of the Aztec diamond, that is, some unit squares are removed from the boundary. By symmetry, if we remove one unit square from the southwestern boundary, then the corresponding unit square on the southeastern boundary also needs to be removed. We label the unit squares on the southwestern boundary of $AD(n)$ by $1,2,\dotsc,n$ from bottom to top. Let $O(n;I)$ be the set of off-diagonally symmetric domino tilings of $AD(n)$ with all unit squares removed from the southwestern boundary \textbf{except} for those labeled $I = \{i_1,\dotsc,i_r\}$, where $1 \leq i_1 < \cdots < i_r \leq n$. Although there is no known or conjectured closed-form formula for $|O(n;I)|$, we summarize below a Pfaffian expression of $|O(n;I)|$ in our earlier work (\cite[Theorems~2 and 4]{Lee23}).
\begin{theorem}[\cite{Lee23}]\label{thm.OffPfaffian}
  Let $A = [a_{i,j}]_{i,j \geq 1}$ be the infinite skew-symmetric matrix whose entries satisfy the following recurrence relation:
  \begin{equation}\label{eq.rec}
  a_{i,j}=
    \begin{cases}
    2, & \text{ if $i=1,j>1$,} \\
    a_{i-1,j} + a_{i,j-1} + a_{i-1,j-1}, & \text{ if $i>1,j>i+1$,} \\
    a_{i-1,j} + a_{i-1,j-1} + 2(-1)^{i-1}, & \text{ if $i>1,j=i+1$,} \\
    -a_{j,i}, & \text{ if $j \leq i$.}
    \end{cases}
  \end{equation}
  Then the number of off-diagonally symmetric domino tilings of $AD(n)$ with some boundary defects (depending on $I$) is given by\footnote{Let $M=[m_{i,j}]$ be a $2n \times 2n$ skew-symmetric matrix. The \emph{Pfaffian of $M$} is defined to be
\begin{equation*}
  \pf(M)= \frac{1}{2^n n!}\sum_{\sigma \in \mathfrak{S}_{2n}} \sgn(\sigma) \prod_{i=1}^{n}m_{\sigma(2i-1), \sigma(2i)},
\end{equation*}
where $\mathfrak{S}_{2n}$ is the symmetric group of $\{1,2,\dotsc,2n\}$ and $\sgn(\sigma)$ is the sign of the permutation $\sigma$.}
  \begin{equation}\label{eq.main1}
    |O(n;I)| = \pf(A_I),
  \end{equation}
  where $A_I$ is the matrix obtained from $A$ by selecting rows and columns indexed by $I$.
\end{theorem}

In the first direction, we consider off-diagonally symmetric domino tilings of $AD(2n-1)$ with one boundary defect, in other words, the index set is $I = [2n-1] \setminus \{k\}$, for some $k=1,2,\dotsc,2n-1$. For convenience, we write $O(2n-1;k)$ for the set $O(2n-1 ; [2n-1] \setminus \{k\})$, for some $k=1,2,\dotsc,2n-1$. Figure \ref{fig.O(5,4)} gives a domino tiling of $O(5,4)$, one can check that the five cells along the vertical diagonal are all assigned $0$ (each such cell contains exactly one complete domino).
\begin{figure}[htb!]
    \centering
    \subfigure[]
    {\label{fig.O(5,4)}\includegraphics[width=0.4\textwidth]{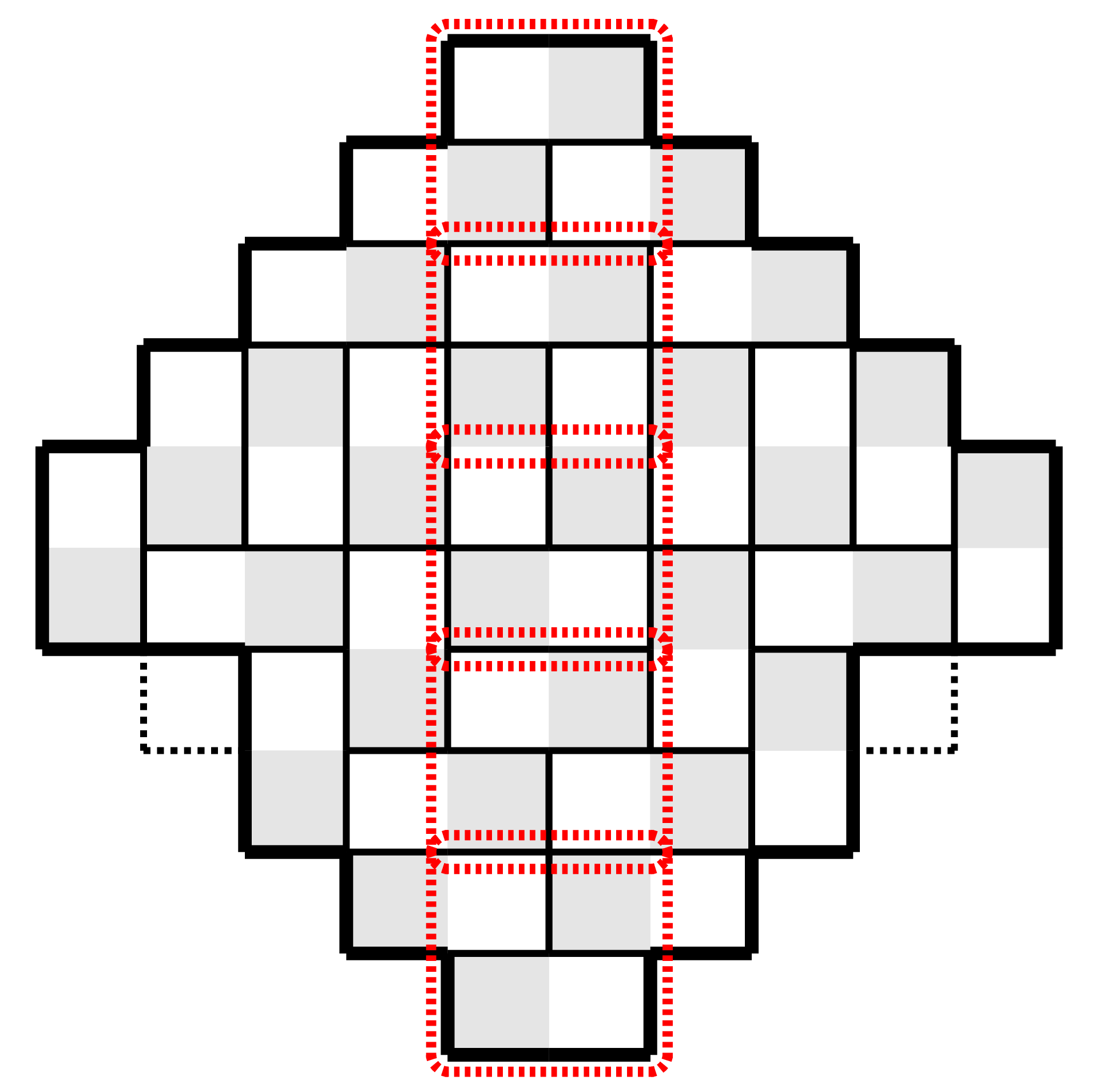}}
    \hspace{5mm}
    \subfigure[]
    {\label{fig.O(6,5)-forced}\includegraphics[width=0.4\textwidth]{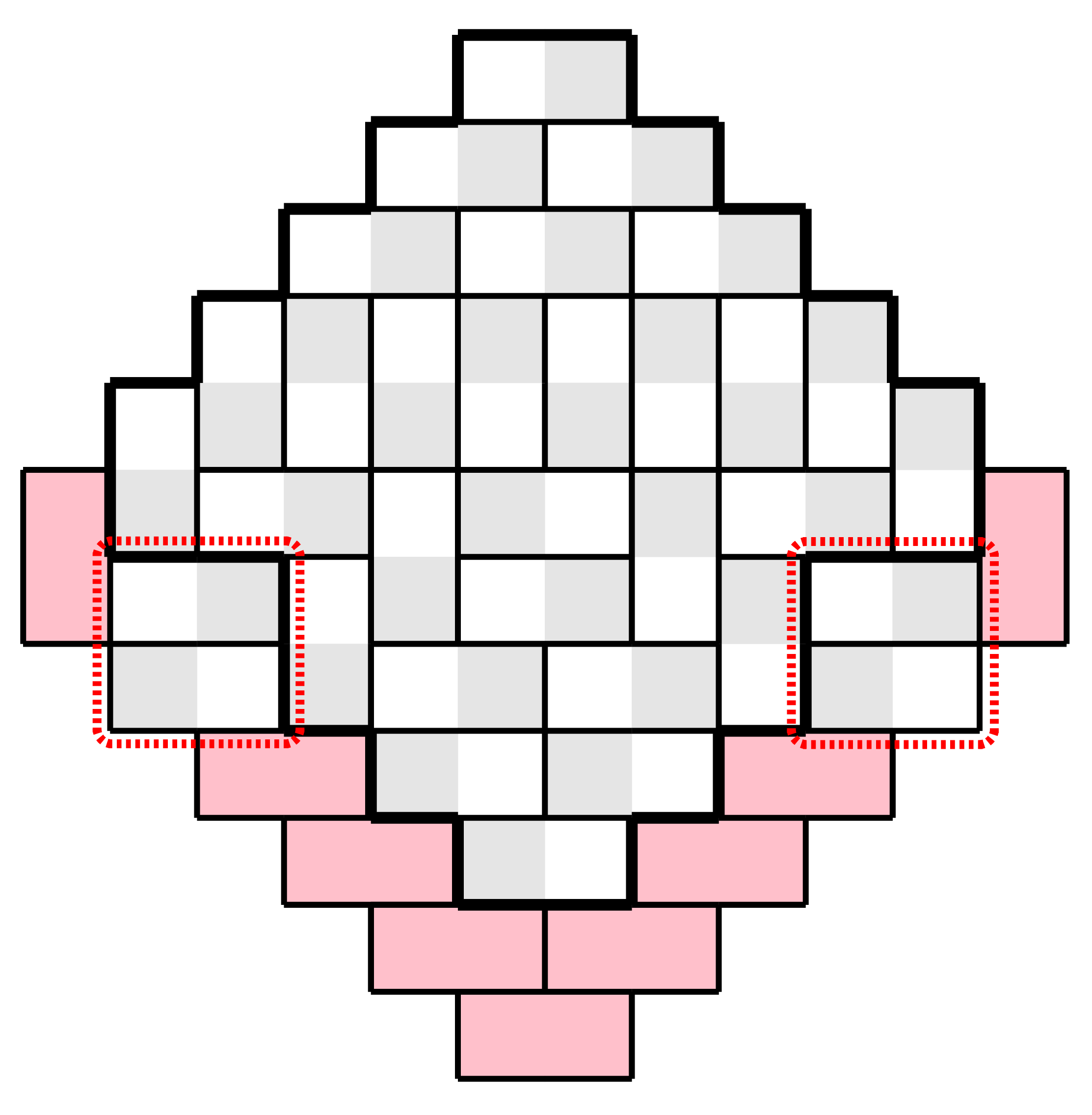}}
    \caption{(a) A domino tiling of $O(5,4)$. (b) An off-diagonally symmetric domino tiling of $AD(6)$ where the cell at the fifth position on the southwestern and southeastern boundary is assigned $1$, and the forced dominoes are colored in pink.}
\end{figure}

Our first main result states the following symmetry property of $|O(2n-1;k)|$:
\begin{theorem}\label{thm.main1}
  The number of off-diagonally symmetric domino tilings of the Aztec diamond of order $2n-1$ with one boundary defect is symmetric in the following sense:
  \begin{equation}\label{eq.offsym}
    |O(2n-1;k)| = |O(2n-1;2n-k)|,
  \end{equation}
  for $k=1,2,\dotsc,n$.
\end{theorem}

We close this subsection by comparing our symmetry property with a recent conjecture proposed by Behrend, Fischer, and Koutschan \cite[Section 8.5]{BFK23}. The connection between alternating sign matrices (ASMs) and domino tilings of the Aztec diamond was initially discovered in \cite{ELKP1, ELKP2}. It was found that the number of domino tilings of the Aztec diamond can be expressed as a weighted enumeration, known as the \textit{$2$-enumeration} of ASMs based on the number of $1$'s or $-1$'s in a matrix. Let $\text{ASM}(n)$ denote the set of ASMs of order $n$. The weighted enumeration is given by
\begin{equation}\label{eq.2-enumasm}
  \M(AD(n)) = \sum_{X \in \text{ASM}(n)}2^{\mathsf{N}_{+}(X)} = \sum_{X \in \text{ASM}(n+1)}2^{\mathsf{N}_{-}(X)},
\end{equation}
where $\mathsf{N}_{+}(X)$ (resp., $\mathsf{N}_{-}(X)$) is the number of $1$'s (resp., $-1$'s) in the matrix $X$.

Off-diagonally symmetric domino tilings of $AD(2n)$ can be regarded as the $2$-enumeration (based on the number of $1$'s) of off-diagonally symmetric ASMs (OSASMs) of order $2n$. We now consider the region $AD(2n)$, if the cell at the position $\tau, (2 \leq \tau \leq 2n)$ on the southwestern boundary (from bottom to top) is assigned $1$, then this leads to a forced domino tiling along the southwestern boundary of $AD(2n)$. Note that a cell that is assigned $1$ contains two complete dominoes---either two horizontal or two vertical dominoes. Due to symmetry, the same occurs on the southeastern boundary. Consequently, the remaining region is $AD(2n-1)$ with one unit square at position $\tau-1$ removed. Figure \ref{fig.O(6,5)-forced} shows an example of a domino tiling of $AD(6)$ with $\tau = 5$ and the forced dominoes are colored pink. The domino tiling of the remaining region is given in Figure \ref{fig.O(5,4)}.

Let $\text{OSASM}(2n)$ denote the set of OSASMs of order $2n$. For a given $X \in \text{OSASM}(2n)$, let $\mathsf{R}(X)$ be the number of nonzero strictly upper triangular entries in $X$ and $\mathsf{T}(X)$ be the column of $1$ in the first row of $X$, ($2 \leq \mathsf{T}(X) \leq 2n$). It is not hard to see that for any matrix $X \in \text{OSASM}(2n)$, $\mathsf{N}_{+}(X) = \mathsf{N}_{-}(X) + 2n$ and $2 \mathsf{R}(X) = \mathsf{N}_{+}(X) + \mathsf{N}_{-}(X)$. Once $\mathsf{R}(X) = \rho$ is given, we have $\mathsf{N}_{+}(X) = \rho +n$ and $\mathsf{N}_{-}(X) = \rho - n$. Set $S(2n;\rho,\tau) = \{ X \in \text{OSASM}(2n) : \mathsf{R}(X) = \rho, \mathsf{T}(X) = \tau\}$. Considering the $2$-enumeration of ASMs mentioned earlier, we have
\begin{equation}\label{eq.2-enumoff}
  2 \cdot |O(2n-1,\tau-1)| = \sum_{\rho \geq 0} |S(2n;\rho,\tau)|2^{\rho+n},
\end{equation}
for $2 \leq \tau \leq 2n$.

The following symmetry property for OSASMs was conjectured in \cite[Conjecture 16]{BFK23}:
\begin{conjecture}[Behrend, Fischer and Koutschan \cite{BFK23}]\label{conj.BFK}
  Given $n$, for any $\rho$ and $\tau$, we have
  \begin{equation}\label{eq.BFKconj}
    |S(2n;\rho,\tau)| = |S(2n;\rho,2n+2-\tau)|.
  \end{equation}
\end{conjecture}
Combining \eqref{eq.offsym} and \eqref{eq.2-enumoff}, our first main result (Theorem~\ref{thm.main1}) is equivalent to the following special case (weighted sum) of Conjecture \ref{conj.BFK}. For $2 \leq \tau \leq n$,
\begin{equation}\label{eq.BFKequiv}
  \sum_{\rho \geq 0} |S(2n;\rho,\tau)|2^{\rho} = \sum_{\rho \geq 0} |S(2n;\rho,2n+2-\tau)|2^{\rho}.
\end{equation}
We hope our method of proving Theorem \ref{thm.main1} could shed new light on studies of the symmetry property for OSASMs.

\subsection{Nearly off-diagonally symmetric domino tilings of the Aztec diamond}\label{sec.nearoff}

In the second direction, we consider \textit{nearly off-diagonally symmetric domino tilings} of $AD(2n-1)$ (without any boundary defects), where only one cell on the diagonal of $AD(2n-1)$ is assigned either $-1$ or $1$, while the other $2n-2$ cells are assigned $0$. According to the position of this specific cell on the diagonal and the value assigned to it, we naturally consider the following cases of nearly off-diagonally symmetric domino tilings of $AD(2n-1)$, denoted by $D^{*}(2n-1;k)$, where the superscript $*$ represents $+,-$ or $\pm$. They are defined as follows.
\begin{itemize}
  \item $D^{+}(2n-1;k)$. The $k$th cell (from bottom to top) on the diagonal is assigned $1$.
  \item $D^{-}(2n-1;k)$. The $k$th cell on the diagonal is assigned $-1$.
  \item $D^{\pm}(2n-1;k)$. The $k$th cell on the diagonal is assigned either $1$ or $-1$. Note that $D^{\pm}(2n-1;k) = D^{+}(2n-1;k) \cup D^{-}(2n-1;k)$.
  \item $D(2n-1) = \bigcup_{k=1}^{2n-1} D^{\pm}(2n-1;k)$.
\end{itemize}

Figure \ref{fig.D+(5,4)} (resp., Figure \ref{fig.D-(5,4)}) illustrates a domino tiling of $D^{+}(5,4)$ (resp., $D^{-}(5,4)$). One can check that the fourth cell (marked by the red dotted box) in Figure \ref{fig.D+(5,4)} (resp., Figure \ref{fig.D-(5,4)}) is assigned $1$ (resp., $-1$), while the other cells on the diagonal are assigned $0$.
\begin{figure}[htb!]
    \centering
    \subfigure[]
    {\label{fig.D+(5,4)}\includegraphics[width=0.4\textwidth]{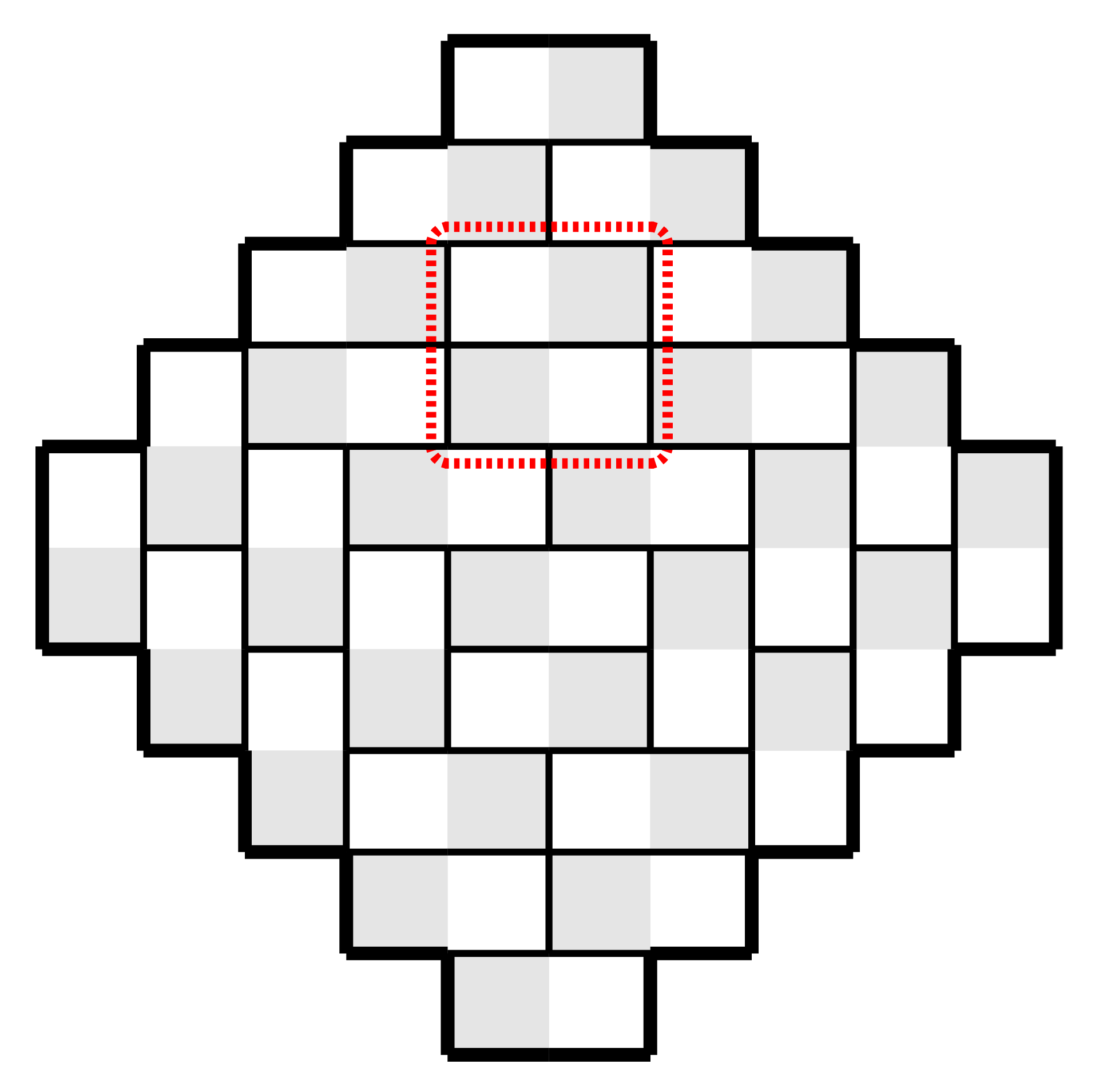}}
    \hspace{5mm}
    \subfigure[]
    {\label{fig.D-(5,4)}\includegraphics[width=0.4\textwidth]{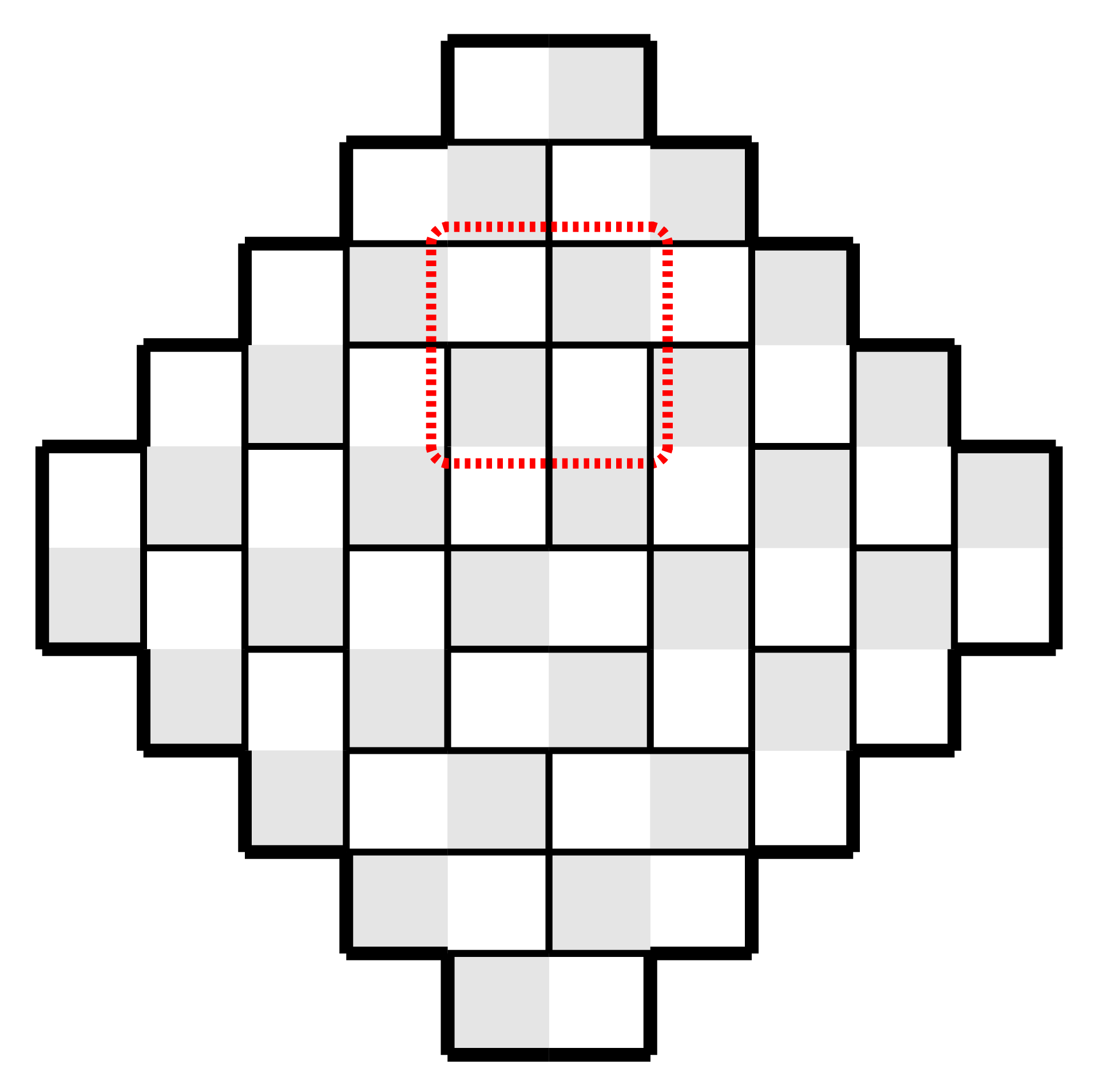}}
    \caption{(a) A domino tiling of $D^{+}(5,4)$. (b) A domino tiling of $D^{-}(5,4)$.}
\end{figure}

In \cite[Page 3]{BFK23}, Behrend, Fischer, and Koutschan introduced the notion of odd-order OSASMs by defining them as diagonally symmetric alternating sign matrices with \textit{maximally-many} $0$'s on the diagonal. Specifically, only one diagonal entry is non-zero while the other diagonal entries are zero. Remarkably, our last case $D(2n-1)$, the set of nearly off-diagonally symmetric domino tilings of $AD(2n-1)$, can be viewed as the $2$-enumeration of OSASMs of order $2n-1$. Although a conjectured product formula for odd-order OSASMs is presented in \cite[Conjecture 15]{BFK23}, our data shows that $|D(2n-1)|$ contains some large prime factors, making it unlikely to have a simple product formula. Instead, we provide a Pfaffian expression of $|D(2n-1)|$, the entries of the Pfaffian satisfy a simple recurrence relation which is similar to the one stated in Theorem \ref{thm.OffPfaffian}.

Let us first define a sequence $(f_n)_{n \geq 1}$ recursively:
\begin{equation}\label{eq.pell}
  f_n = 2f_{n-1} + f_{n-2}, \text{ $(n>2)$, with $f_1 = 2$ and $f_2=4$.}
\end{equation}
The sequence $(f_n)_{n \geq 1}$ is indeed twice the Pell numbers (\cite[A000129]{OEIS}). Let $\mathbf{f}_n = (f_1,f_2,\dotsc,f_n)^{\intercal}$ be the column vector consisting of the first $n$ terms of the sequence $(f_n)_{n \geq 1}$. Define a skew-symmetric matrix $B_{2n}$ of order $2n$ in the following block form:
\begin{equation}\label{eq.matrixB}
  B_{2n} =
  \begin{bmatrix}
    A_{[2n-1]} & \mathbf{f}_{2n-1}  \\
    -(\mathbf{f}_{2n-1})^{\intercal} & 0
  \end{bmatrix},
\end{equation}
where the matrix $A_{[2n-1]}$ is obtained from \eqref{eq.rec} by taking the first $2n-1$ rows and columns.

Now, we are ready to state our second main result.
\begin{theorem}\label{thm.main2}
  The number of nearly off-diagonally symmetric domino tilings of $AD(2n-1)$ is given by
    \begin{equation}\label{eq.nearlyoff}
      |D(2n-1)| = \pf(B_{2n}).
    \end{equation}
  where the matrix $B_{2n}$ is defined above in \eqref{eq.matrixB}.
\end{theorem}

Our data shows that the numbers of domino tilings that we considered in the first direction $|O(2n-1;k)|$, and in the second direction $|D^{*}(2n-1;k)|$ ($*$ represents $+,-$ or $\pm$), do contain some large prime factors. Again, having a product formula seems unlikely. In the remainder of this subsection, we will present simple matrix equations that connect the number of these objects mentioned above.

We view the cardinality of these $2n-1$ sets ($k$ ranges from $1$ to $2n-1$) as column vectors.
\begin{align*}
  \mathbf{O}_{2n-1} &= \left( |O(2n-1;1)|,|O(2n-1;2)|,\dotsc,|O(2n-1;2n-1)| \right)^{\intercal}, \\
  \mathbf{D}^{*}_{2n-1} &= \left( |D^{*}(2n-1;1)|,|D^{*}(2n-1;2)|,\dotsc,|D^{*}(2n-1;2n-1)| \right)^{\intercal},
\end{align*}
where the superscript $*$ represents $+,-$ or $\pm$.

Let $M^{\pm} = [m^{\pm}_{i,j}]_{i,j \geq 1}$ be an infinite matrix, where the entries $m^{\pm}_{i,j} = 2(-1)^{j-1} d_{j-i,i-1}$, and $d_{p,q}$ is the \textit{Delannoy number}\footnote{The Delannoy numbers can be defined recursively $d_{p,q} = d_{p-1,q} + d_{p,q-1} + d_{p-1,q-1}$, with initial values $d_{p,0} = d_{0,q} = 1$ for $p,q \geq 0$, and setting $d_{p,q} = 0$ if $p<0$ or $q<0$. This number counts \textit{Delannoy paths} which are lattice paths going from $(0,0)$ to $(p,q)$ ($p,q \geq 0$), using steps $(1,0)$, $(0,1)$ or $(1,1)$ on the triangular lattice (see for example \cite[Section 2]{BS05} and \cite[Section 2.2]{Lee23}).}. We note that this matrix is indeed upper triangular, as the entries are $0$ when $i>j$. Let $M^{-} = [m^{-}_{i,j}]_{i,j \geq 1}$ be another infinite matrix, where the entries $m^{-}_{i,j} = 2(-1)^{j-1} d_{j-1-i,i-1}$. Define $M^{+} = M^{\pm} - M^{-}$. Our third main result is stated below.
\begin{theorem}\label{thm.main3}
  We have the following three matrix equations
  \begin{equation}\label{eq.matrixod}
    M^{*}_{[2n-1]} \mathbf{O}_{2n-1} = \mathbf{D}^{*}_{2n-1},
  \end{equation}
  where $M^{*}_{[2n-1]}$ is the matrix obtained from $M^{*}$ by taking the first $2n-1$ rows and columns, and the superscript $*$ represents $+,-$ and $\pm$, respectively.
\end{theorem}

We also show the following interesting identities regarding the classes of off-diagonally symmetric Aztec diamonds studied in this paper.
\begin{corollary}\label{cor.identity}
  Given a positive integer $n > 1$. We have the following identities:
\begin{equation}\label{eq.cor1}
  \frac{|O(2n-1;2)|}{|O(2n-1;1)|} = 2n-3 \text{\quad and \quad} \frac{|D^{\pm}(2n-1;2)|}{|D^{\pm}(2n-1;1)|} = 2n-2.
\end{equation}
We also have the following alternating sum identity:
\begin{equation}\label{eq.cor2}
  \sum_{k=1}^{2n-2}(-1)^{k-1} |O(2n-1;k)| = 0.
\end{equation}
\end{corollary} 

The rest of this paper is organized as follows. In Section~\ref{sec.pre}, we review the primary tools that will be used to prove our main results, including the method of non-intersecting lattice paths and a modification of Stembridge's Pfaffian formula for enumerating these non-intersecting lattice paths. In Section~\ref{sec.pair}, we provide several properties of pairs of non-intersecting lattice paths and prove the symmetry property of $|O(2n-1;k)|$ (Theorem \ref{thm.main1}). In Section~\ref{sec.nearly}, we study the nearly off-diagonally symmetric domino tilings of $AD(2n-1)$ and prove Theorem \ref{thm.main2}, Theorem \ref{thm.main3} and Corollary \ref{cor.identity}. In Section~\ref{sec.open}, we propose some conjectures regarding the log-concavity and asymptotic behavior of the number of off-diagonally symmetric domino tilings studied in this paper.

\section{Preliminaries}\label{sec.pre}

We begin this section by reviewing the method of non-intersecting lattice paths (Section~\ref{sec.methodpath}) which turns the enumeration of tilings into the enumeration of families of non-intersecting lattice paths. Subsequently, in Section~\ref{sec.methodnon-int}, we present a Pfaffian formula and its modification for enumerating these non-intersecting lattice paths.

\subsection{The method of non-intersecting lattice paths}\label{sec.methodpath}

The method of non-intersecting lattice paths (or simply non-intersecting paths) (see~\cite[Section~3.1]{Propp15}) is a powerful technique used to enumerate domino or lozenge tilings. The key idea is to view a tiling as a family of non-intersecting paths\footnote{We say two paths are \textit{non-intersecting} if they do not pass through the same vertex. A family of paths is non-intersecting if any two of the paths are non-intersecting.} and then try to enumerate these families of non-intersecting paths by, for instance, the Lindstr{\"o}m-Gessel-Viennot theorem (\cite{Lin73} and \cite{GV85}) or Stembridge's Pfaffian generalization (\cite{Stem90}).

In general, there is a bijection (see \cite[Section 2]{LRS01}) between the set of domino tilings of a region $R$ on the square lattice and families of non-intersecting Delannoy paths with specific starting and ending points, which are determined by $R$. This bijection has been extensively applied in various studies concerning domino tilings of the Aztec diamond, see for example \cite{EF05},\cite{BvL13},\cite{Lai16}, and \cite{Lee23}. We describe this bijection below.

Given $AD(n)$ with the checkerboard coloring (mentioned in Section~\ref{sec.off}) on it. We mark the midpoint of the left edge of each black unit square, join these midpoints by edges, and form a subgraph $\mathcal{AD}(n)$ of the triangular lattice (see Figure~\ref{fig.AD(6)} for the graph $\mathcal{AD}(6)$ drawn in dotted edges). For convenience, we choose a coordinate system on the triangular lattice by fixing a lattice point as the origin and letting the positive $x$-axis (resp., $y$-axis) be a lattice line pointing southeast (resp., northeast). Consequently, the edges that are parallel to the $x$-axis (resp., $y$-axis) are oriented southeast (resp., northeast), while the edges parallel to the line $y=x$ are oriented east (such orientations are illustrated in Figure \ref{fig.AD(6)} with red edges).
\begin{figure}[htb!]
    \centering
    \subfigure[]
    {\label{fig.AD(6)}\includegraphics[height=0.5\textwidth]{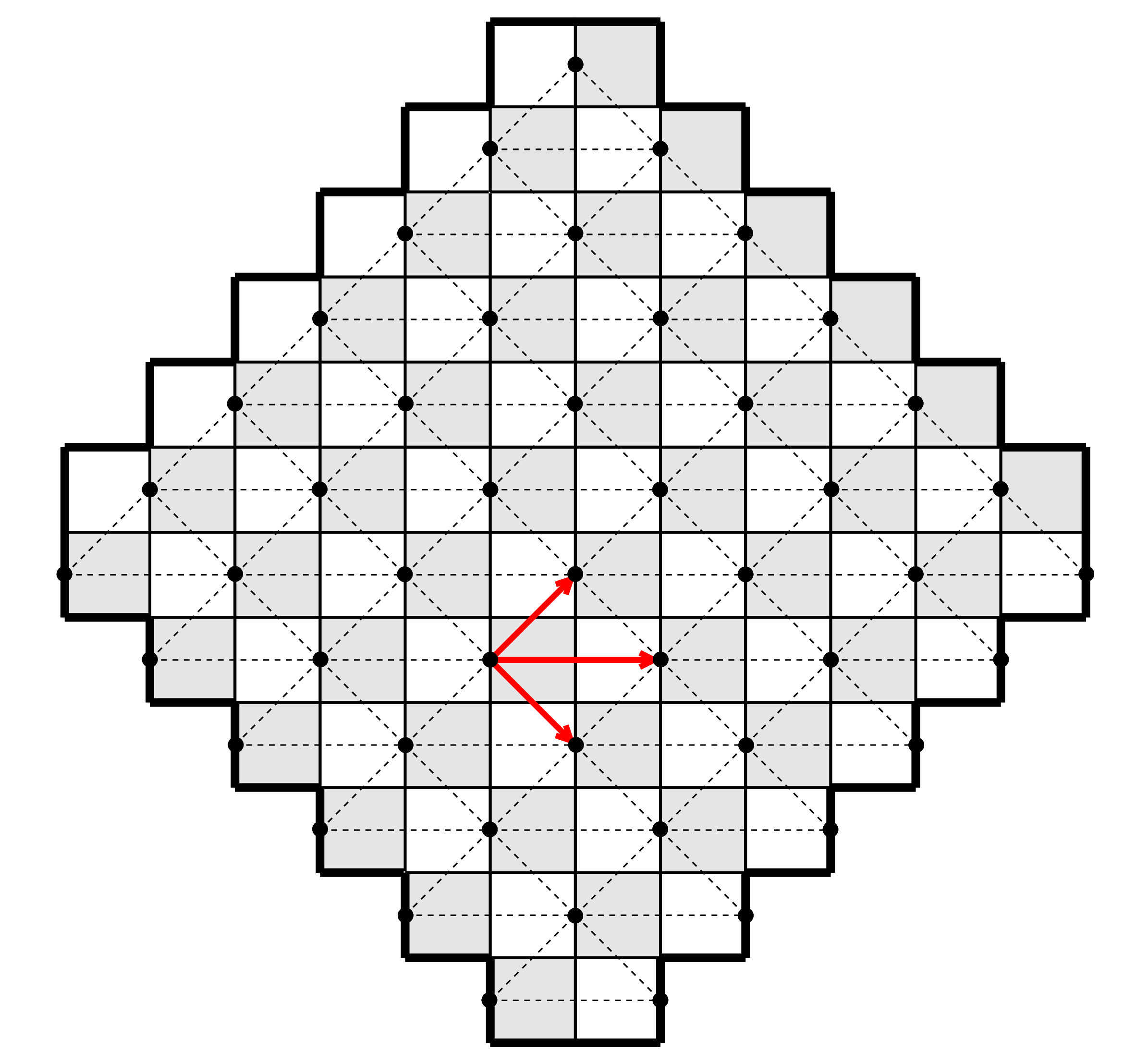}}
    \hspace{5mm}
    \subfigure[]
    {\label{fig.tri-half}\includegraphics[height=0.5\textwidth]{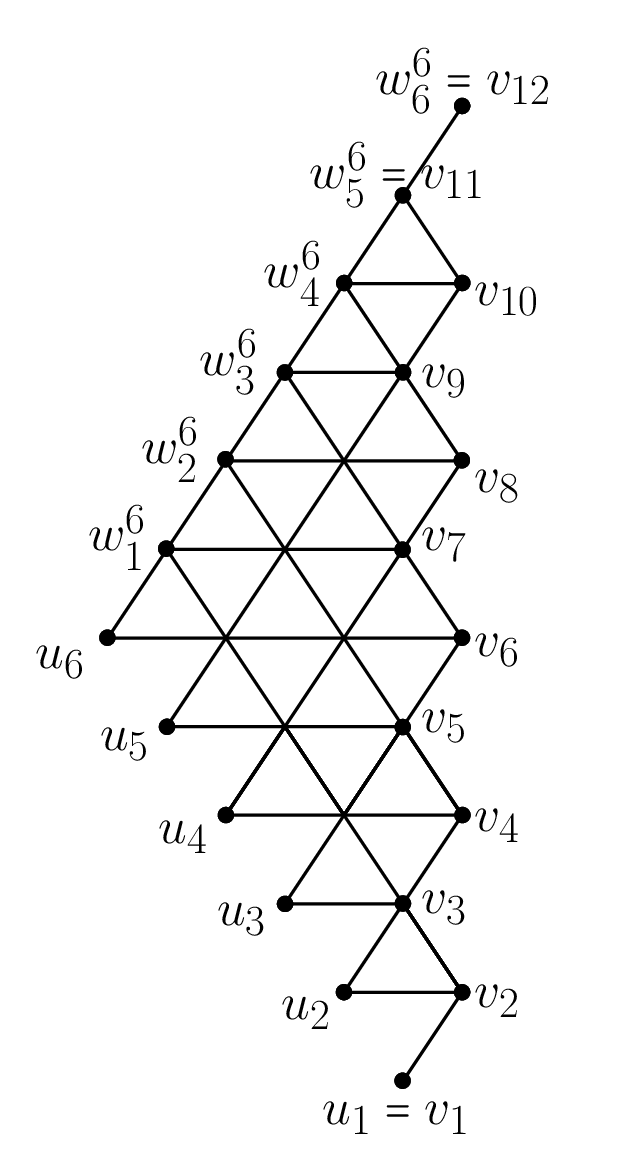}}
    \caption{(a) The region $AD(6)$ and the underlying graph $\mathcal{AD}(6)$ drawn in dotted edges. (b) The graph $\mathcal{D}(6)$, the left half part of $\mathcal{AD}(6)$.}
    \label{fig.tri-graph}
\end{figure}

The bijection works as follows. Given a domino tiling, map each horizontal domino with the left unit square colored black to a $(1,1)$ step on the triangular lattice; map each vertical domino with the top (resp., bottom) unit square colored black to a $(1,0)$ (resp., $(0,1)$) step on the triangular lattice. Note that no step of the lattice paths corresponds to horizontal dominoes in which the right unit square is black. Figure~\ref{fig.bijection} shows two examples of this bijection; we would like to point out that since the fourth unit square is removed in Figure~\ref{fig.O(5,4)path}, there is no path starting at $u_4$.
\begin{figure}[htb!]
    \centering
    \subfigure[]
    {\label{fig.O(5,4)path}\includegraphics[height=0.4\textwidth]{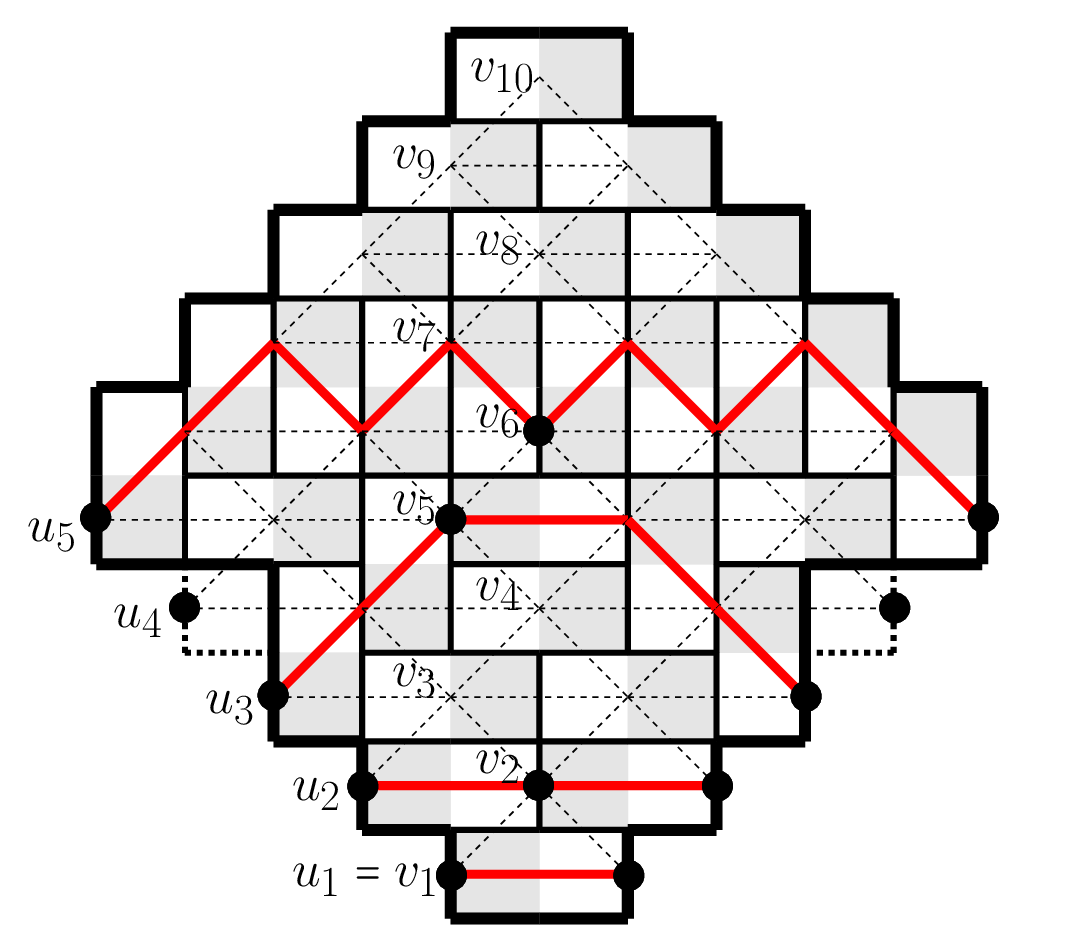}}
    \hspace{5mm}
    \subfigure[]
    {\label{fig.D+(5,4)path}\includegraphics[height=0.4\textwidth]{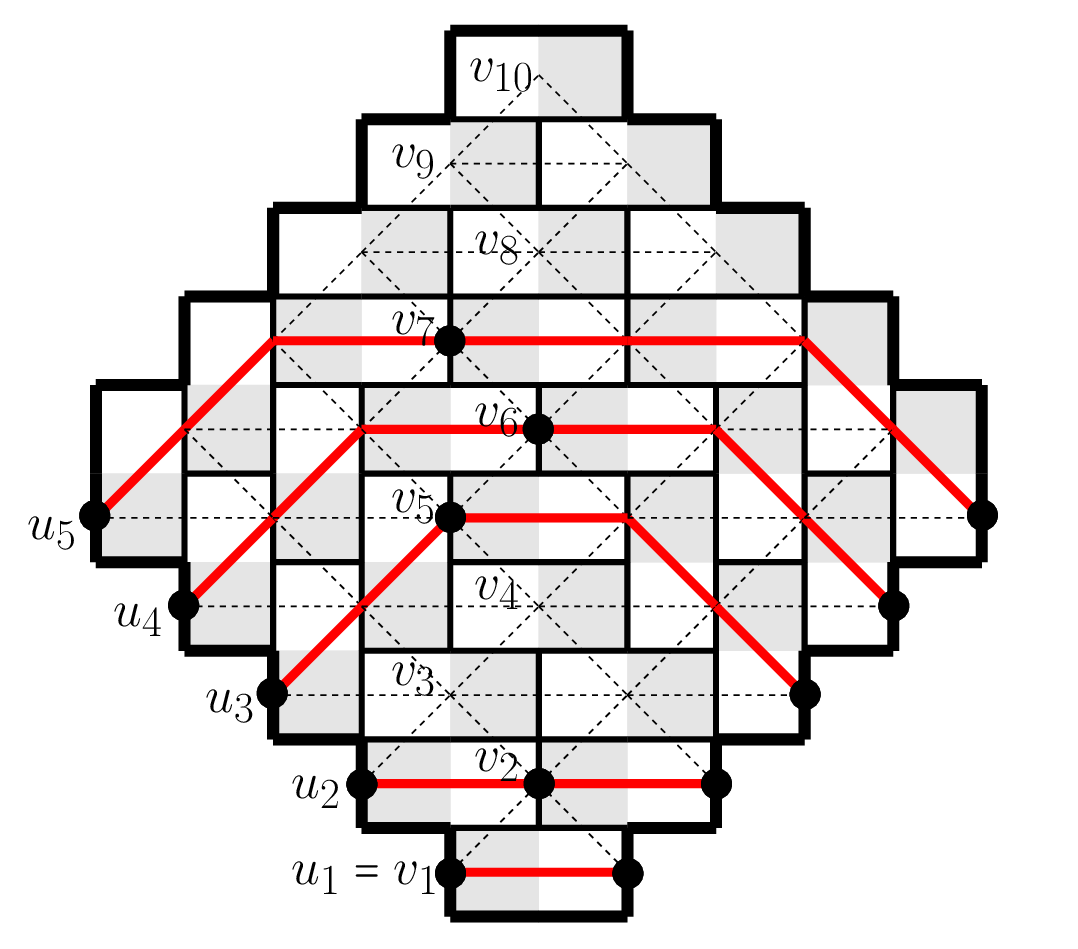}}
    \caption{(a) The non-intersecting path corresponding to the domino tiling given in Figure~\ref{fig.O(5,4)}. (b) The non-intersecting path corresponding to the domino tiling given in Figure~\ref{fig.D+(5,4)}.}
    \label{fig.bijection}
\end{figure}

We now return to off-diagonally symmetric domino tilings of the Aztec diamond. Due to symmetry, we can focus on the left half part of the Aztec diamond. Let $\mathcal{D}(n)$ be the subgraph of $\mathcal{AD}(n)$ containing the left half part of it; see Figure~\ref{fig.tri-half} for an example when $n=6$. On the graph $\mathcal{D}(n)$, let $U = \{u_1,u_2,\dotsc,u_n\}$ be the collection of points on the southwestern boundary (from bottom to top) and $V = \{v_1,v_2,\dotsc,v_{2n}\}$ the collection of $2n$ points along the right boundary (from bottom to top). We also let $W=\{w_1^{n},\dotsc,w_n^{n}\}$ be the collection of points on the northwestern boundary (from bottom to top).

In the previous work \cite[Lemma 7]{Lee23}, the author showed that the $k$th cell (from bottom to top) on the vertical diagonal is assigned $0$ if and only if either $v_{2k-1},v_{2k} \in V_0$ or $v_{2k-1},v_{2k} \notin V_0$, where $ V_0 \subset V$ is the set of ending points of the corresponding family of Delannoy paths. We call a pair of points $v_k^{*}=\{v_{2k-1},v_{2k}\}$ a \textit{doublet}, and let $V^{*}=\{v_1^{*},\dotsc v_n^{*}\}$ be the collection of these doublets. As a consequence, given an index set $I = \{i_1,\dotsc,i_r\}$, where $1 \leq i_1 < \cdots < i_r \leq n$, off-diagonally symmetric domino tilings of $O(n;I)$ are in one-to-one correspondence with $r$-tuples of non-intersecting Delannoy paths in $\mathcal{D}(n)$ connecting $U=\{u_{i_1},\dotsc,u_{i_r}\}$ and $V^{*}$ (two points of a doublet are both the ending points of paths or neither the ending points of paths). We refer the interested reader to \cite[Section 2.3]{Lee23} for more details.

For nearly off-diagonally symmetric domino tilings of $AD(2n-1)$, the cell on the vertical diagonal that is assigned $-1$ or $1$ can also be characterized by ending points of its corresponding non-intersecting paths. This follows immediately from the contrapositive statement of the characterization of a cell assigned $0$ mentioned in the previous paragraph (\cite[Lemma 7]{Lee23}). In other words, the $k$th cell on the vertical diagonal is assigned $-1$ or $1$ if and only if either $v_{2k-1}$ or $v_{2k}$ is the ending point of the corresponding path.

When the $k$th cell on the vertical diagonal is assigned $1$ (it contains either two horizontal or two vertical dominoes), we can count this case by viewing it as twice the number of the case when that cell contains two horizontal dominoes. It is easy to see that the $k$th cell on the vertical diagonal contains two horizontal dominoes if and only if $v_{2k-1}$ is the ending point of the corresponding path while $v_{2k}$ is not (see the fourth cell in Figure \ref{fig.D+(5,4)path} as an example).

For the last case when the $k$th cell on the vertical diagonal is assigned $-1$, it is given by
\begin{equation*}
    D^{-}(2n-1;k) = D^{\pm}(2n-1;k) \setminus D^{+}(2n-1;k).
\end{equation*}
Therefore, we are able to enumerate domino tilings of the set $D^{*}(2n-1;k)$ ($*$ represents $+,-$ or $\pm$) by enumerating its corresponding families of non-intersecting paths. We summarize the above discussion in the following lemma.
\begin{lemma}\label{lemma.cell}
  On the graph $\mathcal{D}(2n-1)$, let $U=\{u_{1},\dotsc,u_{2n-1}\}$ be the set of fixed starting points of paths, $V^{*}=\{v_1^{*},\dotsc v_{2n-1}^{*}\}$ the set of doublets, and $\overline{V}=\{v_{\ell}\}$ the fixed ending point, for some $\ell$. Let $\mathscr{P}_0(U,\overline{V} \cup V^{*})$ be the set of $(2n-1)$-tuples of non-intersecting paths connecting $U$ with $\overline{V} \cup V^{*}$ in which $2n-2$ paths end at some doublets while the remaining one path must end at $v_{\ell}$. Then we have
  \begin{align*}
    |D^{\pm}(2n-1;k)| & = |\mathscr{P}_0(U,\{v_{2k}\} \cup V^{*})| + |\mathscr{P}_0(U,\{v_{2k-1}\} \cup V^{*})|, \\
    |D^{+}(2n-1;k)| & = 2 \cdot |\mathscr{P}_0(U,\{v_{2k-1}\} \cup V^{*})|, \\
    |D^{-}(2n-1;k)| & = |\mathscr{P}_0(U,\{v_{2k}\} \cup V^{*})| - |\mathscr{P}_0(U,\{v_{2k-1}\} \cup V^{*})|,
  \end{align*}
  for $k=1,\dotsc,2n-1$.
\end{lemma}

\subsection{Enumeration of non-intersecting lattice paths}\label{sec.methodnon-int}

Let $\mathcal{G}$ be a locally finite, connected, directed acyclic graph, equipped with a weight function $\wt$ assigning elements from some commutative ring to each of its edges. Given three ordered sets $U = \{u_1,\dotsc,u_n\}$, $\overline{V} = \{\overline{v_1},\dotsc,\overline{v_m}\}, (m \leq n)$ and $V =  \{v_1,\dotsc,v_s\}, (s \geq n-m)$ of vertices of $\mathcal{G}$. We write $\mathscr{P}^{\pi}(U, \overline{V} \cup V)$ for $n$-tuples of paths $(p_1,\dotsc,p_n)$ with the \textit{connection type} $\pi \in \mathfrak{S}_n$ (the symmetric group of $[n]$), where
\begin{itemize}
  \item for $i=1,2,\dotsc,m$, the path $p_i$ connects $u_{\pi(i)} \in U$ with $\overline{v_i} \in \overline{V}$, and
  \item for $i=m+1,m+2,\dotsc,n$, the path $p_i$ connects $u_{\pi(i)}$ with some point $v_{j_i} \in V$, with $1 \leq j_{m+1} < j_{m+2} < \cdots < j_n \leq s$.
\end{itemize}
We write $\mathscr{P}_0^{\pi}(U,\overline{V} \cup V)$ for the subset of $\mathscr{P}^{\pi}(U,\overline{V} \cup V)$ consisting of $n$-tuples of non-intersecting paths. If we do not specify the connection type, then we write $\mathscr{P}(U,\overline{V} \cup V)$ (resp., $\mathscr{P}_0(U,\overline{V} \cup V)$) for the union of $n$-tuples of (resp., non-intersecting) paths in $\mathscr{P}^{\pi}(U,\overline{V} \cup V)$ (resp., $\mathscr{P}^{\pi}_0(U,\overline{V} \cup V)$) over all connection types $\pi \in \mathfrak{S}_n$. Two ordered sets of vertices $U = \{u_1,\dotsc,u_n\}$ and $\overline{V} \cup V = \{\overline{v_1},\dotsc,\overline{v_m},v_1,\dotsc,v_s\}$ of $\mathcal{G}$ mentioned above are said to be \textit{compatible} if the only way to form $n$-tuples of non-intersecting paths is when the connect type $\pi = \mathsf{id}$.

The weight of a path $p$ is defined to be $\wt(p) = \prod_{e}\wt(e)$, where the product is over all edges $e$ of the path $p$. For an $n$-tuple of paths $P=(p_1,\dotsc,p_n)$, the weight is given by $\wt(P) = \prod_{i=1}^{n} \wt(p_i)$. Given a set of ($n$-tuples of) paths $\mathscr{P}$, we write $GF(\mathscr{P}) = \sum_{P \in \mathscr{P}} \wt(P)$ for the weighted sum of all the elements in $\mathscr{P}$.

The total weight of $n$-tuples of non-intersecting paths in $\mathscr{P}_0(U,\overline{V} \cup V)$ has been studied by Lindstr{\"o}m \cite{Lin73}, and later by Gessel and Viennot \cite{GV85} when $V=\emptyset$ and $m=n$. Stembridge \cite[Theorem~3.2]{Stem90} generalized their results to the case when $V$ is non-empty and the sets $U$ and $\overline{V} \cup V$ are compatible. We state below a more general version ($U$ and $\overline{V} \cup V$ are not necessarily compatible) given by Ciucu and Krattenthaler in \cite[Theorem 5]{CK11}:
\begin{theorem}[Stembridge~\cite{Stem90}; Ciucu, Krattenthaler~\cite{CK11}]\label{thm.Stem2}
   On a locally finite, connected, directed acyclic graph $\mathcal{G}$ with the setting mentioned above and assume $n+m$ is even. Then we have
   \begin{equation}\label{eq.SK1}
     \sum_{\pi \in \mathfrak{S}_n} \sgn(\pi) GF \left( \mathscr{P}_{0}^{\pi}(U,\overline{V} \cup V) \right) = (-1)^{\binom{m}{2}}\pf
   \begin{bmatrix}
      Q & H \\
      -H^{\intercal} & 0_m
      \end{bmatrix},
   \end{equation}
   where $Q=[Q_{V}(u_i,u_j)]_{1 \leq i,j \leq n}$ is a skew-symmetric matrix of order $n$, $Q_{V}(u_i,u_j)$ is the total weight of pairs of non-intersecting paths $(p_i,p_j)$, where $p_i$ connects $u_i$ with some $v_{\ell_1}$ and $p_j$ connects $u_j$ with some $v_{\ell_2}$, with $\ell_1 < \ell_2$; and where $H=[GF(\mathscr{P}(u_i,\overline{v_j}))]_{1 \leq i \leq n, 1 \leq j \leq m}$ is the rectangular matrix, $\mathscr{P}(u_i,\overline{v_j})$ is the set of paths going from $u_i$ to $\overline{v_j}$. Here, $0_m$ is the zero matrix of order $m$.

   In particular, if $U$ and $\overline{V} \cup V$ are compatible, then we have
   \begin{equation}\label{eq.SK2}
     GF \left( \mathscr{P}_{0}^{\mathsf{id}}(U,\overline{V} \cup V) \right) = (-1)^{\binom{m}{2}}\pf
   \begin{bmatrix}
      Q & H \\
      -H^{\intercal} & 0_m
      \end{bmatrix}.
   \end{equation}
\end{theorem}
\begin{remark}\label{remark.Q}
  $Q_V(u_i,u_j)$ is obtained from summing over all possible pairs of non-intersecting paths whose ending points are in $V$, it can be written as the following double summation (\cite[Eq. 3.1]{Stem90}).
  \begin{equation}\label{eq.Q}
      Q_V(u_i,u_j) = \sum_{1 \leq \ell_1 < \ell_2 \leq s} \det
      \begin{bmatrix}
      GF\left( \mathscr{P}(u_i,v_{\ell_1}) \right) & GF\left( \mathscr{P}(u_i,v_{\ell_2}) \right) \\
      GF\left( \mathscr{P}(u_j,v_{\ell_1}) \right) & GF\left( \mathscr{P}(u_j,v_{\ell_2}) \right)
      \end{bmatrix}.
  \end{equation}
\end{remark}

From the discussion in Section \ref{sec.methodpath}, the families of non-intersecting paths that we would like to enumerate have one additional condition---the paths end at some doublets in $V^{*} = \{v_1^{*},\dotsc,v_{d}^{*}\}$. We write $\mathscr{P}_{0}^{\pi}(U,\overline{V} \cup V^{*})$ for the subset of $\mathscr{P}_{0}^{\pi}(U,\overline{V} \cup V)$ where the two elements in $v^{*}_{\ell}$ are both the ending points of paths or neither the ending points of paths, for all $\ell$. Note that $\mathscr{P}_{0}(U,\overline{V} \cup V^{*}) = \emptyset$ if $|U|-|\overline{V}|=n-m$ is odd.

Theorem \ref{thm.Stem2} does not apply directly in our situation because of this additional condition on the ending points of paths. This requires a modification of Theorem \ref{thm.Stem2}. In the previous paper \cite[Lemma~11]{Lee23}, the author provided such a modification to the case when there are no fixed ending points of paths ($\overline{V} = \emptyset$). We would like to point out that, when $\overline{V} = \emptyset, (m=0)$, the matrix in \eqref{eq.SK1} reduces to the matrix $[Q_V(u_i,u_j)]_{1 \leq i,j \leq n}$ of order $n$. The modification is to replace this matrix by $[Q_{V^{*}}(u_i,u_j)]_{1 \leq i,j \leq n}$, where $Q_{V^{*}}(u_i,u_j)$ is the total weight of pairs of non-intersecting paths $(p_i,p_j)$, where $p_i$ starts from $u_i$, $p_j$ starts from $u_j$, and the ending points of $p_i$ and $p_j$ are both in a doublet $v^{*}_{\ell}$, for some $\ell=1,\dotsc,d$. Using a similar argument to the one that led to \eqref{eq.Q}, we have
  \begin{equation}\label{eq.Q*}
      Q_{V^{*}}(u_i,u_j) = \sum_{1 \leq \ell \leq d} \det
      \begin{bmatrix}
      GF\left( \mathscr{P}(u_i,v_{2\ell -1}) \right) & GF\left( \mathscr{P}(u_i,v_{2\ell}) \right) \\
      GF\left( \mathscr{P}(u_j,v_{2\ell -1}) \right) & GF\left( \mathscr{P}(u_j,v_{2\ell}) \right)
      \end{bmatrix}.
  \end{equation}

The same modification can be extended directly to Theorem \ref{thm.Stem2}, the statement is presented in the following lemma. We will omit the proof here since it follows from the same arguments that prove \cite[Lemma~11]{Lee23}.
\begin{lemma}\label{lemma.mod2}
   On a locally finite, connected, directed acyclic graph $\mathcal{G}$ with the same setting as Theorem \ref{thm.Stem2} except that $V$ is replaced by the set of doublets $V^{*}= \{v_1^{*},\dotsc,v_{d}^{*}\}$, and assume $n+m$ is even. Then we have
   \begin{equation}\label{eq.mod21}
     \sum_{\pi \in \mathfrak{S}_n} \sgn(\pi) GF \left( \mathscr{P}_{0}^{\pi}(U,\overline{V} \cup V^{*}) \right) = (-1)^{\binom{m}{2}}\pf
   \begin{bmatrix}
      Q^{*} & H \\
      -H^{\intercal} & 0_m
      \end{bmatrix},
   \end{equation}
   where $Q^{*}=[Q_{V^{*}}(u_i,u_j)]_{1 \leq i,j \leq n}$ is a skew-symmetric matrix of order $n$, $Q_{V^{*}}(u_i,u_j)$ is given in \eqref{eq.Q*}, and where $H=[GF(\mathscr{P}(u_i,\overline{v_j}))]_{1 \leq i \leq n, 1 \leq j \leq m}$ is the rectangular matrix. Here, $0_m$ is the zero matrix of order $m$.

   In particular, if $U$ and $\overline{V} \cup V^{*}$ are compatible, then we have
   \begin{equation}\label{eq.mod22}
     GF \left( \mathscr{P}_{0}^{\mathsf{id}}(U,\overline{V} \cup V^{*}) \right) = (-1)^{\binom{m}{2}}\pf
   \begin{bmatrix}
      Q^{*} & H \\
      -H^{\intercal} & 0_m
      \end{bmatrix}.
   \end{equation}
\end{lemma}

\section{Pairs of non-intersecting lattice paths}\label{sec.pair}

For the rest of the paper, we continue with the same labeling of points and the same orientation of edges on the graph $\mathcal{D}(n)$ mentioned in Section~\ref{sec.methodpath}. We will use the same notations of paths given in Section~\ref{sec.methodnon-int} and work on the graph with all the edge weights equal to $1$.

We provide an intermediate result (Proposition~\ref{prop.Oreverse}) in Section~\ref{sec.onedefect}. The proof of Proposition~\ref{prop.Oreverse} will be elaborated in subsequent sections, where we show various properties of pairs of non-intersecting lattice paths. These properties are discussed in Sections~\ref{sec.transinv} (translation invariant), \ref{sec.recurrence} (recurrence relation) and \ref{sec.involutory} (involutory matrix). Furthermore, the proof of our first main result (Theorem \ref{thm.main1}) will be given in Section~\ref{sec.symmetry}.

\subsection{Enumeration of off-diagonally symmetric domino tilings of the Aztec diamond with one boundary defect}\label{sec.onedefect}

We remind the reader that $O(n;k)$ is the set of off-diagonally symmetric domino tilings of $AD(n)$ with the $k$th unit square removed from the southwestern boundary (along with the one from the southeastern boundary). Previously, we introduced the column vector $\mathbf{O}_{n}$, which includes the cardinality of the sets $O(n;k)$ for $k=1,\dotsc,n$. Let $\mathsf{rev}$ be an operator that reverses the order of elements in a vector, then
\begin{equation*}
  \mathbf{O}^{\mathsf{rev}}_{n} = \left( |O(n;n)|,|O(n;n-1)|,\dotsc,|O(n;1)| \right)^{\intercal}.
\end{equation*}
From the discussion in Section~\ref{sec.off}, if $n$ is even, then $O(n;k) = \emptyset$, and thus $\mathbf{O}_{n}$ is the zero vector.

We provide the following intermediate result which plays an important role in proving our first main result (Theorem~\ref{thm.main1}).
\begin{proposition}\label{prop.Oreverse}
  There exists an $n \times n$ matrix $R_{n}=[(-1)^{n+j}r^{n}_{i,j}]$ such that
  \begin{equation}\label{eq.RO}
    R_{n} \mathbf{O}_{n} = \mathbf{O}^{\mathsf{rev}}_{n},
  \end{equation}
  where $r^{n}_{i,j} = Q_{V^{*}}(u_j,w^{n}_{i})$. In particular, the matrix $R_{n}$ has the following properties:
  \begin{enumerate}
    \item $R_{n}$ is upper triangular with $r^{n}_{i,i} = 1$, for $i=1,\dotsc,n$,
    \item (translation invariant) $r^{n}_{i,j} = r^{n}_{i+1,j+1}$, for $1 \leq i, j \leq n-1$,
    \item (three-term recurrence relation) $r^{n}_{1,j} = r^{n}_{1,j-1} + r^{n-1}_{1,j} + r^{n-1}_{1,j-1}$, for $n \geq 3$ and $2 \leq j \leq n-1$, and
    \item (involution) $(R_n)^{-1} = R_n$.
  \end{enumerate}
\end{proposition}
\begin{proof}
  We will prove the existence of the matrix $R_{n}$ and the first property here, the proof of the other properties of $R_{n}$ will be presented in Sections \ref{sec.transinv}, \ref{sec.recurrence} and \ref{sec.involutory}.

  Domino tilings of $O(n;k)$ can be enumerated in two ways. The first way follows from Theorem~\ref{thm.OffPfaffian} by taking the index set $I = [n] \setminus \{k\}$, then we obtain
  \begin{equation}\label{eq.R1}
    |O(n,k)| = \pf(A_{I}).
  \end{equation}
  The second way is to reflect the entire Aztec diamond region and the tiling across the horizontal diagonal, then the boundary defect appears on the northwestern (and northeastern) boundary. Figure~\ref{fig.O(5,4)reverse} shows an example and its corresponding non-intersecting paths. According to the bijection mentioned in Section~\ref{sec.methodpath}, removing the $k$th unit square from the northwestern boundary of $AD(n)$ creates the new starting point $w_{n+1-k}^n$ on the graph $\mathcal{D}(n)$. As a consequence, $O(n;k)$ is in one-to-one correspondence with $\mathscr{P}_0^{\mathsf{id}}(U \cup \{w_{n+1-k}^n\}, V^{*})$ on the graph $\mathcal{D}(n)$ (it is obvious that $U \cup \{w_{n+1-k}^n\}$ and $V^{*}$ are compatible, the only connection type is $\mathsf{id}$).
\begin{figure}[htb!]
    \centering
    \subfigure[]
    {\label{fig.O(5,4)rev}\includegraphics[height=0.4\textwidth]{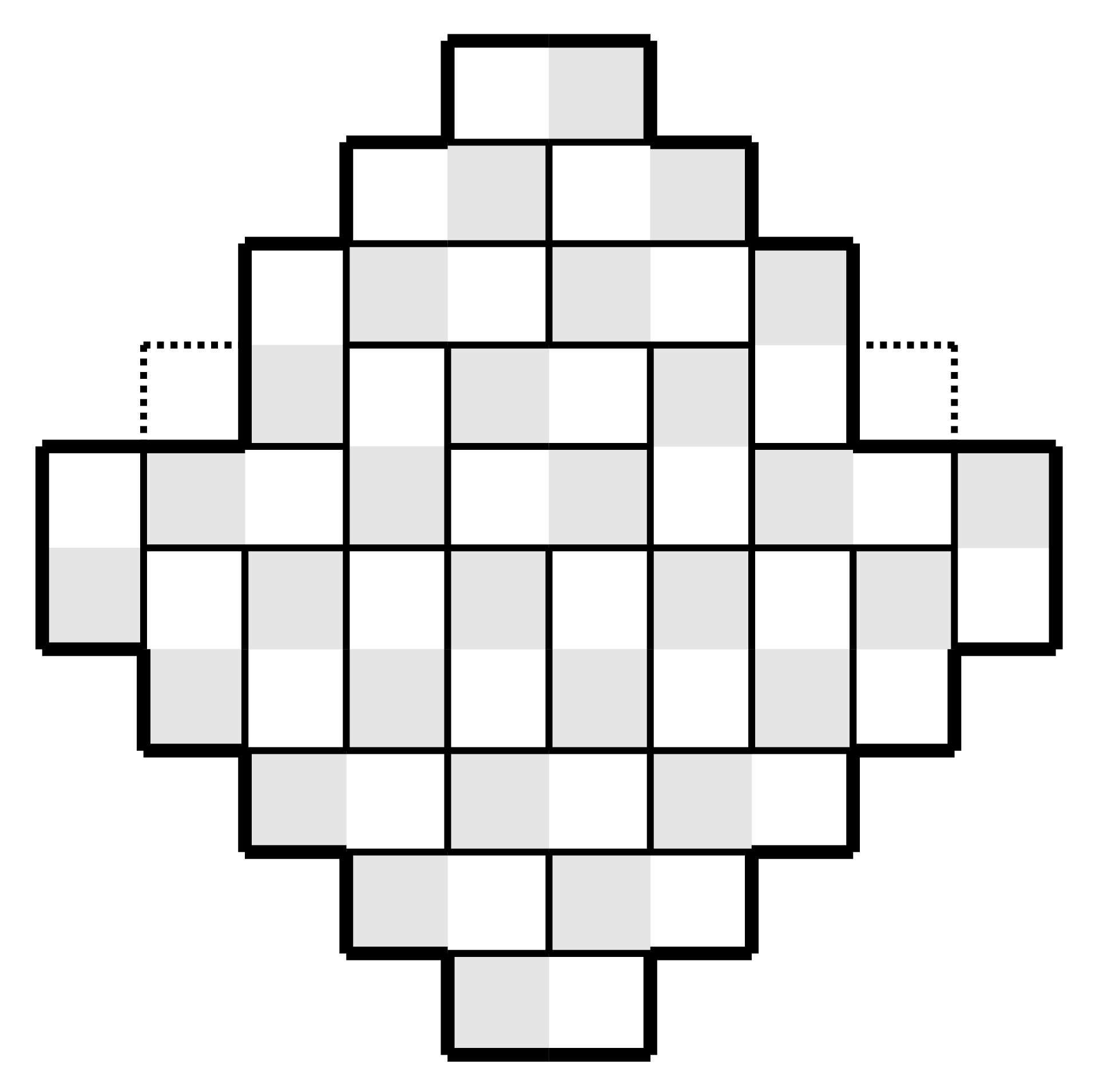}}
    \hspace{5mm}
    \subfigure[]
    {\label{fig.O(5,4)revpath}\includegraphics[height=0.4\textwidth]{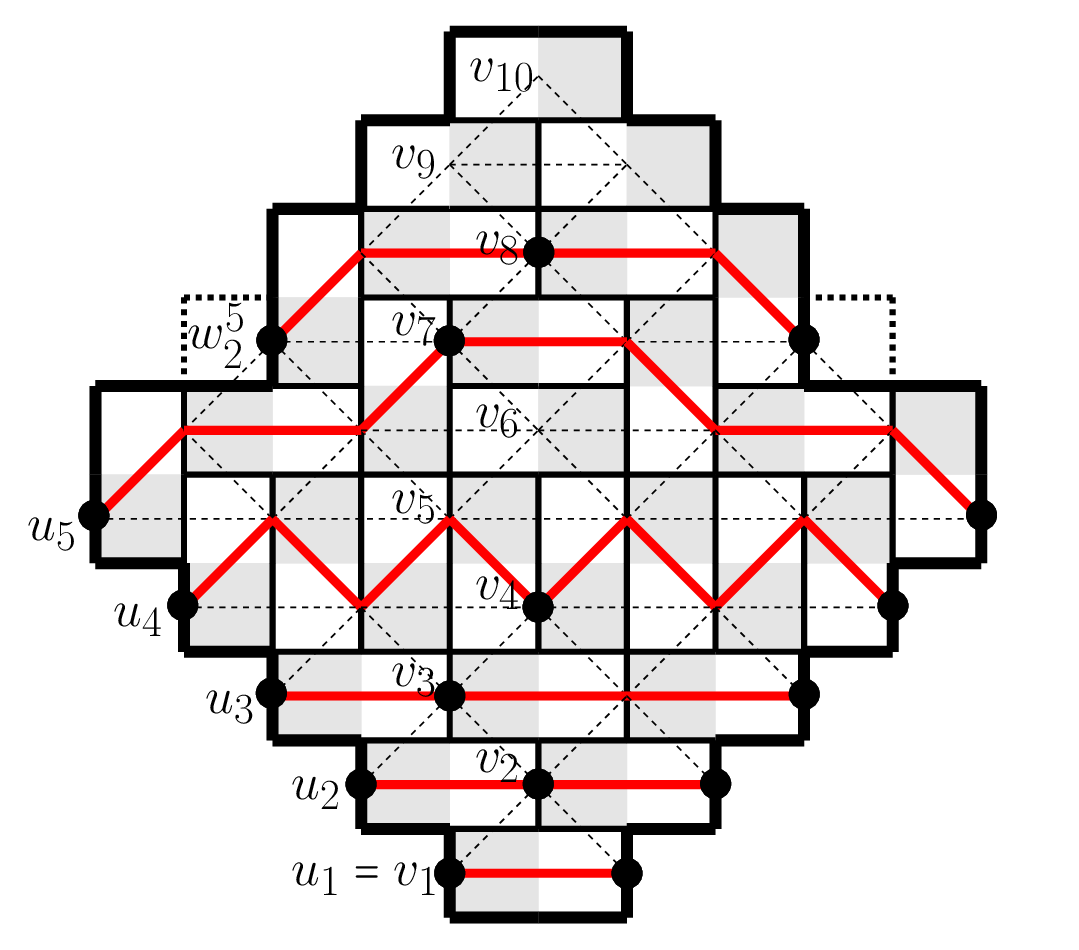}}
    \caption{(a) The reflection of the domino tiling given in Figure \ref{fig.O(5,4)} with respect to the horizontal diagonal. (b) The corresponding non-intersecting paths from the domino tiling in Figure \ref{fig.O(5,4)rev}.}
    \label{fig.O(5,4)reverse}
\end{figure}

  We then apply the special case of Lemma~\ref{lemma.mod2} by taking the graph $\mathcal{G} = \mathcal{D}(n)$, the sets $\overline{V} = \emptyset$, $V^{*}=\{v_1^{*},\dotsc,v_n^{*}\}$ and $U=\{u_1,\dotsc,u_n,w_{n+1-k}^n\}$ to obtain $|O(n;k)|$. However, to prove this proposition, it is more convenient to obtain $|O(n;n+1-k)|$ (by substituting $k$ with $n+1-k$). As a result, we have the following Pfaffian expression. For $k=1,\dotsc,n$,
  \begin{equation}\label{eq.R2}
    |O(n;n+1-k)| = GF(\mathscr{P}_0^{\mathsf{id}}(U, V^{*})) = \pf
       \begin{bmatrix}
          Q^{*} & \mathbf{x}_k^{n} \\
          -(\mathbf{x}_k^{n})^{\intercal} & 0
       \end{bmatrix},
  \end{equation}
  where $Q^{*}=[Q_{V^{*}}(u_i,u_j)]_{1 \leq i,j \leq n}$ is the matrix $A_{[n]}$ given in \eqref{eq.rec}, and $\mathbf{x}_k^{n}$ is a column vector which corresponds to the term involving the new starting point, the $i$th component of $\mathbf{x}_k^{n}$ is given by $Q_{V^{*}}(u_i,w_{k}^n)$, $i=1,\dotsc,n$.

  We remind the reader below a cofactor-like expansion for $\pf(A)$ along the last column of $A=[a_{i,j}]_{1 \leq i,j \leq n}$ (see for instance \cite[Proposition~2.3]{IW06}):
  \begin{equation}\label{eq.cofactorpf}
    \pf(A) = \sum_{\ell=1}^{n-1} (-1)^{n+\ell-1} a_{\ell,n} \pf(A_{\hat{\ell},\hat{n}}),
  \end{equation}
  where $A_{\hat{\ell},\hat{n}}$ is the matrix obtained from $A$ by deleting rows and columns indexed by $\ell$ and $n$. It is worth noting that if we delete the rows and columns indexed by $k$ and $n+1$ from the matrix on the right-hand side of \eqref{eq.R2}, then we obtain the matrix $A_{[n] \setminus \{k\}}$. The Pfaffian of this matrix gives, in fact, $|O(n,k)|$ by \eqref{eq.R1}. Now, we expand the Pfaffian in \eqref{eq.R2} along the last column, which results in
  \begin{equation}\label{eq.R3}
    |O(n;n+1-k)| = \pf
       \begin{bmatrix}
          Q^{*} & \mathbf{x}_k^{n} \\
          -(\mathbf{x}_k^{n})^{\intercal} & 0
       \end{bmatrix}
       = \sum_{\ell=1}^{n} (-1)^{(n+1)+\ell-1} Q_{V^{*}}(u_{\ell},w_{k}^n) |O(n;\ell)|,
  \end{equation}
  for $k=1,\dotsc,n$.

  It is clear that the $n$ equations in \eqref{eq.R3} are equivalent to the matrix equation of the form $R_n \mathbf{O}_n = \mathbf{O}_n^{\mathsf{rev}}$, where the $(k,\ell)$-entry of $R_n$ (the $k$th row of $R_n$ corresponds to the $k$th equation in \eqref{eq.R3}) is given by $(-1)^{n+\ell} Q_{V^{*}}(u_{\ell},w_{k}^n)$. This verifies \eqref{eq.RO}.

  On the graph $\mathcal{D}(n)$ (see Figure~\ref{fig.tri-half} again), it is easy to see that the path starting from $u_j,(2 \leq j \leq n)$ can end at points from $v_2$ to $v_{2j}$, the path starting from $u_1$ can end at $v_1$ or $v_2$. On the other hand, the path starting from $w_i^n,(1 \leq i \leq n)$ can reach points from $v_{2i}$ to $v_{2n}$.

  In the case when $i=j$, $r_{i,i}^n = Q_{V^{*}}(u_i,w_i^n) = 1$ since there is only one pair of non-intersecting paths: one path connects $u_i$ with $v_{2i-1}$ and the other path connects $w_i^n$ with $v_{2i}$. When $i>j$, $r_{i,j}^n = Q_{V^{*}}(u_j,w_i^n) = 0$ since the ending points that the two paths can reach are not in a doublet, such a pair of non-intersecting paths does not exist. This implies that the matrix $R_{n}$ is upper triangular. This completes the proof of the first part of Proposition~\ref{prop.Oreverse}.
\end{proof}

\subsection{Translation invariant lemmas}\label{sec.transinv}

Consider the graph $\mathcal{\overline{D}}(n)$ obtained from $\mathcal{D}(n)$ by deleting all the $u_i$'s and all their incident edges, the vertices on the southwestern boundary are labeled by $x_1$ to $x_n$ from bottom to top; see Figure~\ref{fig.tri-halfcut} for $\mathcal{\overline{D}}(6)$. We note that on the graph $\mathcal{D}(n)$, the set of doublets is $V^{*}=\{v_{\ell}^{*} = \{v_{2\ell-1},v_{2\ell}\} : 1 \leq \ell \leq n\}$. However, on the graph $\mathcal{\overline{D}}(n)$, the set of doublets becomes $V^{*}=\{v_{\ell}^{*} = \{v_{2\ell-1},v_{2\ell}\} : 2 \leq \ell \leq n\}$ since $v_1$ has been deleted.

\begin{figure}[htb!]
    \centering
    \subfigure[]
    {\label{fig.tri-halfcut}\includegraphics[height=0.4\textwidth]{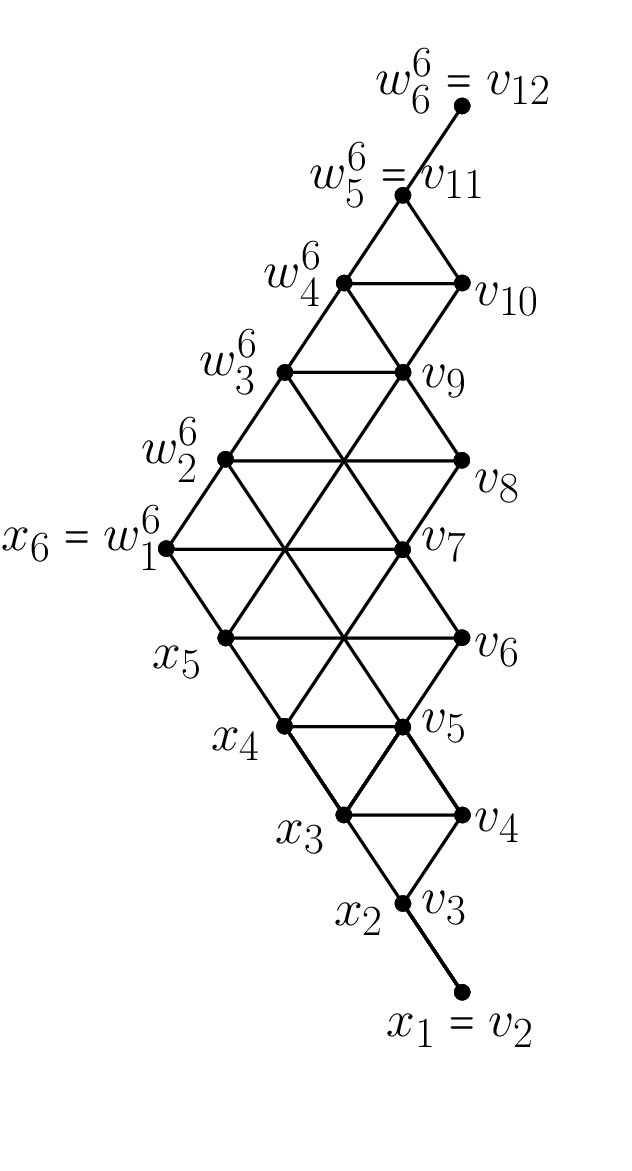}}
    \hspace{10mm}
    \subfigure[]
    {\label{fig.transinv}\includegraphics[height=0.4\textwidth]{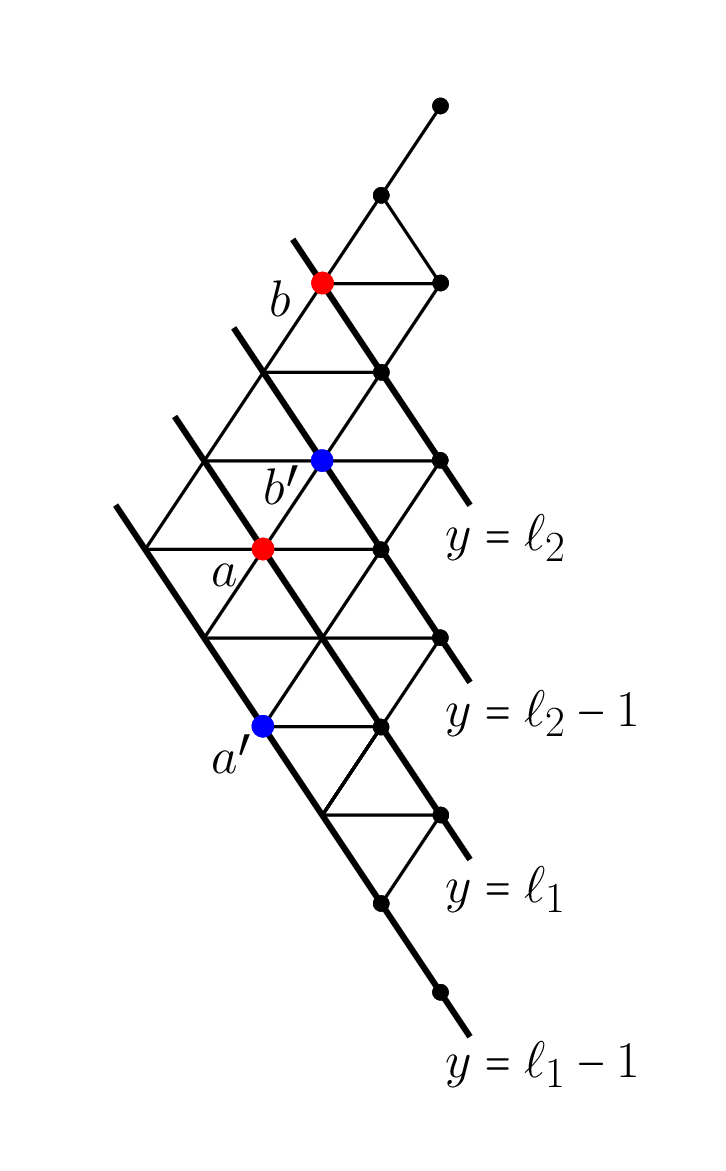}}
    \hspace{10mm}
    \subfigure[]
    {\label{fig.tri-partition}\includegraphics[height=0.4\textwidth]{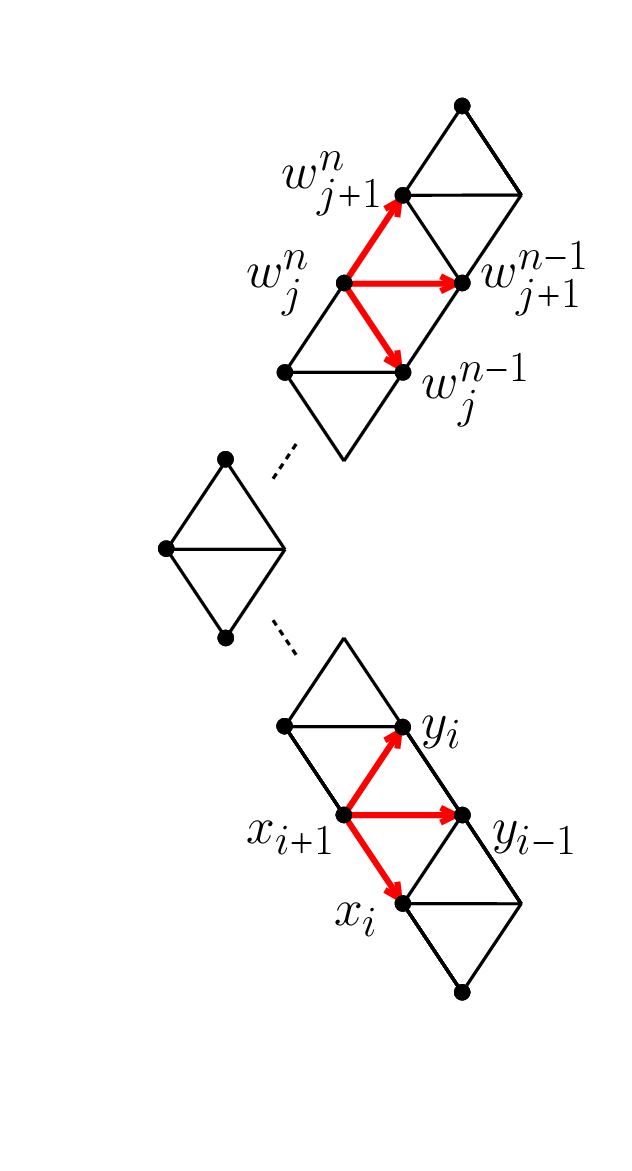}}
    \caption{(a) The graph $\mathcal{\overline{D}}(6)$. (b) The scenario described in Lemma~\ref{lemma.trans1}. (c) An illustration of partitioning paths in the proof Lemma~\ref{lemma.trans2}.}
    \label{fig.tricut-graph}
\end{figure}

The first translation invariant lemma stated below is similar to \cite[Lemma 13]{Lee23}.
\begin{lemma}\label{lemma.trans1}
  On the graph $\mathcal{\overline{D}}(n)$, let $a=(p,\ell_1)$ and $b=(q,\ell_2)$ be two distinct points on the lattice line $y=\ell_1$ and $y=\ell_2$, respectively. Let $a^{\prime}$ and $b^{\prime}$ be the points obtained from $a$ and $b$ by shifting ``downward'' one lattice line, that is, $a^{\prime} = (p+1,\ell_1-1)$ and $b^{\prime} = (q+1,\ell_2-1)$, provided that $a^{\prime}$ and $b^{\prime}$ are still contained in $\mathcal{\overline{D}}(n)$ (see Figure~\ref{fig.transinv}). Then
  \begin{equation}\label{eq.trans1}
    Q_{V^{*}}(a,b) = Q_{V^{*}}(a^{\prime},b^{\prime}).
  \end{equation}
\end{lemma}
\begin{proof}
   Notice that the number of paths going from $a$ to $v_{\ell+2}$ is the same as the number of paths going from $a^{\prime}$ to $v_{\ell}$ for all $\ell$. Namely, $|\mathscr{P}(a,v_{\ell+2})| = |\mathscr{P}(a^{\prime},v_{\ell})|$ for all $\ell$. Similarly, $|\mathscr{P}(b,v_{\ell+2})| = |\mathscr{P}(b^{\prime},v_{\ell})|$ for all $\ell$.

   Since $a^{\prime}$ and $b^{\prime}$ are not on the northwestern boundary of $\mathcal{\overline{D}}(n)$ (otherwise, $a$ and $b$ would not be contained in $\mathcal{\overline{D}}(n)$), there is no path starting from $a^{\prime}$ and $b^{\prime}$ could reach $v_{2n-1}$ (and also $v_{2n}$). Hence, $|\mathscr{P}(a^{\prime},v_{2n -1})| = |\mathscr{P}(b^{\prime},v_{2n -1})| = 0$. By a similar argument that $a$ and $b$ are not on the southwestern boundary of $\mathcal{\overline{D}}(n)$, there is no path starting from $a$ or $b$ that could reach $v_3$. Hence, $|\mathscr{P}(a,v_3)| = |\mathscr{P}(b,v_3)| = 0$.

   From the above discussion, we use \eqref{eq.Q*} and shift the index $\ell$ by $1$ to obtain identity:
   \begin{align*}
     Q_{V^{*}}(a^{\prime},b^{\prime}) &= \sum_{2 \leq \ell \leq n-1} \det
          \begin{bmatrix}
          |\mathscr{P}(a^{\prime},v_{2\ell -1})| & |\mathscr{P}(a^{\prime},v_{2\ell})| \\
          |\mathscr{P}(b^{\prime},v_{2\ell -1})| & |\mathscr{P}(b^{\prime},v_{2\ell})|
          \end{bmatrix} =
          \sum_{2 \leq \ell \leq n-1} \det
          \begin{bmatrix}
          |\mathscr{P}(a,v_{2\ell +1})| & |\mathscr{P}(a,v_{2\ell+2})| \\
          |\mathscr{P}(b,v_{2\ell +1})| & |\mathscr{P}(b,v_{2\ell+2})|
          \end{bmatrix} \\
       & = \sum_{3 \leq \ell \leq n} \det
          \begin{bmatrix}
          |\mathscr{P}(a,v_{2\ell -1})| & |\mathscr{P}(a,v_{2\ell})| \\
          |\mathscr{P}(b,v_{2\ell -1})| & |\mathscr{P}(b,v_{2\ell})|
          \end{bmatrix} = Q_{V^{*}}(a,b),
   \end{align*}
   as desired.
\end{proof}

Our next translation invariant lemma shows that $Q_{V^{*}}(a,b)$ is almost invariant under a ``clockwise rotation'' of two points $a$ and $b$, where $a$ lies on the southwestern boundary while $b$ lies on the northwestern boundary of the graph $\mathcal{\overline{D}}(n)$. The lemma is presented below.
\begin{lemma}\label{lemma.trans2}
On the graph $\mathcal{\overline{D}}(n)$, for $n \geq 2$ and $1 \leq i \leq n-1$, we have
  \begin{equation}\label{eq.Q*-invDbar}
    Q_{V^{*}}(x_i,w^n_j) =
    \begin{cases}
        Q_{V^{*}}(x_{i+1},w^n_{j+1}) + (-1)^i, & \text{if $j=1$,} \\
        Q_{V^{*}}(x_{i+1},w^n_{j+1}), & \text{if $2 \leq j \leq n-1$.}
    \end{cases}
  \end{equation}
\end{lemma}
\begin{proof}
  We proceed with the following three cases.
  \begin{itemize}
    \item[Case 1:] $i=1$. \\
        The path starting from $x_1$ has no option but to end at $v_2$ (the same point as $x_1$), which does not belong to any doublet. So, no matter where the second path starts from, it is impossible to form a pair of non-intersecting paths whose ending points are in a doublet. This implies that $Q_{V^{*}}(x_1,w^n_j) = 0$ for $1 \leq j \leq n-1$.

        On the other hand, the path starting from $x_2$ must end at $v_3$, allowing the second path to end at $v_4$. This happens only when the second path starts from $w_2^{n}$ (paths starting from $w_j^{n},j>2$ can not reach $v_4$) and there is only one such pair of non-intersecting paths. Then we obtain
        \begin{equation*}
            Q_{V^{*}}(x_2,w^n_{j+1}) =
            \begin{cases}
                1, & \text{if $j=1$,} \\
                0, & \text{if $2 \leq j \leq n-1$.}
            \end{cases}
          \end{equation*}
        Thus, \eqref{eq.Q*-invDbar} holds for $i=1$.

    \item[Case 2:] $j=n-1$. \\
        The idea is similar to Case 1. To form such a pair of non-intersecting paths, the path starting from $w_{n-1}^{n}$ must end at $v_{2n-2}$, allowing the other path to end at $v_{2n-3}$. This happens when the latter path starts from $x_{n-1}$ (paths starting from $x_i,i<n-1$ can not reach $v_{2n-3}$). Thus,
        \begin{equation*}
            Q_{V^{*}}(x_{i},w^n_{n-1}) =
            \begin{cases}
                1, & \text{if $i=n-1$,} \\
                0, & \text{if $1 \leq i \leq n-2$.}
            \end{cases}
        \end{equation*}

        On the other hand, the path starting from $w_{n}^{n}$ has no option but to end at $v_{2n}$, the other path must start from $x_n$ and end at $v_{2n-1}$. Thus,
        \begin{equation*}
            Q_{V^{*}}(x_{i+1},w^n_{n}) =
            \begin{cases}
                1, & \text{if $i=n-1$,} \\
                0, & \text{if $1 \leq i \leq n-2$.}
            \end{cases}
        \end{equation*}
        This shows that \eqref{eq.Q*-invDbar} holds for $j=n-1$.

    \item[Case 3:] $2 \leq i \leq n-1$ and $1 \leq j \leq n-2$. \\
        We will prove it by induction on $n$, with the base case being $n=3$, this implies that $i=2$ and $j=1$. On the left-hand side of \eqref{eq.Q*-invDbar}, $Q_{V^{*}}(x_{2},w^3_{1})=2$ since there are two such pairs of non-intersecting paths, where one path connects $x_2$ with $v_3$, the other path connects $w_1^{3}$ with $v_4$ in two ways. On the right-hand side of \eqref{eq.Q*-invDbar}, $Q_{V^{*}}(x_{3},w^3_{2}) + (-1)^2 = 1+1 = 2$ since there is only one pair of non-intersecting paths, where one path connects $x_3$ with $v_3$, the other path connects $w^3_2$ with $v_4$. Hence, \eqref{eq.Q*-invDbar} holds for the base case $n=3$.

        Now, we assume the statement holds for $n-1$, that is, for $2 \leq i \leq n-2$ and $1 \leq j \leq n-3$,
        \begin{equation}\label{eq.Q*induc}
          Q_{V^{*}}(x_i,w^n_j) = Q_{V^{*}}(x_{i+1},w^n_{j+1}) + (-1)^i \delta_{j,1},
        \end{equation}
        where $\delta$ denotes the Kronecker delta function. Due to Cases 1 and 2, the induction hypothesis can be extended to $1 \leq i \leq n-2$ and $1 \leq j \leq n-2$. It would be more convenient to work with the equivalent hypothesis by shifting the index $i$ by $1$:
        \begin{equation}\label{eq.Q*inducnew}
          Q_{V^{*}}(x_{i-1},w^n_j) = Q_{V^{*}}(x_{i},w^n_{j+1}) + (-1)^{i-1} \delta_{j,1},
        \end{equation}
        for $2 \leq i \leq n-1$ and $1 \leq j \leq n-2$.

        The idea is to partition paths starting from $w_j^n$ into the disjoint union of three sets, based on three possible first steps of a path; see Figure \ref{fig.tri-partition}. In other words, for $1 \leq j \leq n-2$, we have
        \begin{equation*}
          \mathscr{P}(w_j^n,v) = \mathscr{P}(w_{j+1}^n,v) \cup \mathscr{P}(w_{j+1}^{n-1},v) \cup \mathscr{P}(w_j^{n-1},v),
        \end{equation*}
        and
        \begin{equation*}
          |\mathscr{P}(w_j^n,v)| = |\mathscr{P}(w_{j+1}^n,v)| + |\mathscr{P}(w_{j+1}^{n-1},v)| + |\mathscr{P}(w_j^{n-1},v)|,
        \end{equation*}
        for any $v \in V$. Fix $x_i$, using \eqref{eq.Q*} and the linearity of determinants, we obtain
        \begin{equation}\label{eq.Q*induc1}
          Q_{V^{*}}(x_i,w^n_j) = Q_{V^{*}}(x_i,w^n_{j+1}) + Q_{V^{*}}(x_i,w^{n-1}_j) + Q_{V^{*}}(x_i,w^{n-1}_{j+1}).
        \end{equation}

        Following a similar argument, we also obtain the identity when partitioning paths starting from $x_{i+1}$ into the disjoint union of three sets. For $2 \leq i \leq n-1$,
        \begin{equation}\label{eq.Q*induc2}
          Q_{V^{*}}(x_{i+1},w^n_{j+1}) = Q_{V^{*}}(x_i,w^n_{j+1}) + Q_{V^{*}}(y_{i-1},w^{n}_{j+1}) + Q_{V^{*}}(y_{i},w^{n}_{j+1}).
        \end{equation}

        By Lemma \ref{lemma.trans1} and the induction hypothesis \eqref{eq.Q*inducnew}, we have
        \begin{align}
           Q_{V^{*}}(y_{i},w^{n}_{j+1}) & = Q_{V^{*}}(x_i,w^{n-1}_{j}) \label{eq.Q*induc3} \\
           Q_{V^{*}}(y_{i-1},w^{n}_{j+1}) & = Q_{V^{*}}(x_{i-1},w^{n-1}_{j}) = Q_{V^{*}}(x_{i},w^{n-1}_{j+1}) + (-1)^{i-1}\delta_{j,1}, \label{eq.Q*induc4}
        \end{align}
        for $2 \leq i \leq n-1$ and $1 \leq j \leq n-2$.

        Substituting the terms on the right-hand side of \eqref{eq.Q*induc2} by \eqref{eq.Q*induc3} and \eqref{eq.Q*induc4}, and then comparing them with the terms on the right-hand side of \eqref{eq.Q*induc1}, we finally obtain the desired identity \eqref{eq.Q*-invDbar}.
  \end{itemize}
This completes the proof of Lemma~\ref{lemma.trans2}.
\end{proof}

Now, we are ready to prove the translation invariant property in Proposition \ref{prop.Oreverse} (2), which states that $r_{i,j}^n = r_{i+1,j+1}^n$ for $1 \leq i,j \leq n-1$. Recall that $r_{i,j}^n = Q_{V^{*}}(u_j,w^n_i)$.

\begin{proof}[Proof of Proposition~\ref{prop.Oreverse} (2)]
On the graph $\mathcal{D}(n)$, it suffices to show that 
\begin{equation*}
    Q_{V^{*}}(u_j,w^n_i) = Q_{V^{*}}(u_{j+1},w^n_{i+1}), \text{ for $1 \leq i,j \leq n-1$.}
\end{equation*}
When $j=1$, using a similar argument in Case 1 of the proof of Lemma~\ref{lemma.trans2}, we obtain
  \begin{equation*}
    Q_{V^{*}}(u_1,w^n_i) = \delta_{i,1} = Q_{V^{*}}(u_{2},w^n_{i+1}), \text{ for $1 \leq i \leq n-1$.}
  \end{equation*}

When $2 \leq j \leq n-1$, using a similar idea of partitioning paths as in Case 3 of the proof of Lemma~\ref{lemma.trans2}, the paths starting from $u_j$ are partitioned into the disjoint union of two sets, based on two possible first steps of a path. In other words, for $2 \leq j \leq n-1$, we have
    \begin{equation}\label{eq.partitionu}
      \mathscr{P}(u_j,v) = \mathscr{P}(x_j,v) \cup \mathscr{P}(x_{j-1},v),
    \end{equation}
  for any $v \in V$. Similarly, fix a point $w_i^n$, we have
    \begin{equation}\label{eq.Q*Dpartition}
      Q_{V^{*}}(u_j,w^n_i) = Q_{V^{*}}(x_j,w^n_{i}) + Q_{V^{*}}(x_{j-1},w^{n}_i).
    \end{equation}

Finally, we obtain
  \begin{align*}
    Q_{V^{*}}(u_j,w^n_i) & = Q_{V^{*}}(x_j,w^n_{i}) + Q_{V^{*}}(x_{j-1},w^{n}_i) & \text{(by \eqref{eq.Q*Dpartition})} \\
     & = \left(Q_{V^{*}}(x_{j+1},w^n_{i+1}) + (-1)^{j}\delta_{i,1}\right) +  \left(Q_{V^{*}}(x_{j},w^{n}_{i+1}) + (-1)^{j-1}\delta_{i,1}\right) & \text{(by \eqref{eq.Q*-invDbar})} \\
     & = Q_{V^{*}}(u_{j+1},w^n_{i+1}), & \text{(by \eqref{eq.Q*Dpartition})}
  \end{align*}
  for $2 \leq j \leq n-1$ and $1 \leq i \leq n-1$. This completes the proof of Proposition~\ref{prop.Oreverse} (2).
\end{proof}

\subsection{A three-term recurrence relation}\label{sec.recurrence}

In the following lemma, we present a three-term recurrence relation of the number of pairs of non-intersecting paths $Q_{V^{*}}(x_i,w^n_2)$ on the graph $\mathcal{\overline{D}}(n)$. Note that the second path in a pair has the fixed starting point $w_2^n$.

\begin{lemma}\label{lemma.rec}
  On the graph $\mathcal{\overline{D}}(n)$, for $n \geq 4$ and $3 \leq i \leq n-1$, we have
  \begin{equation}\label{eq.lemmarec}
    Q_{V^{*}}(x_i,w^n_2) = Q_{V^{*}}(x_i,w^{n-1}_2) + Q_{V^{*}}(x_{i-1},w^{n-1}_2) + Q_{V^{*}}(x_{i-1},w^n_2).
  \end{equation}
\end{lemma}

\begin{proof}
  Following the same idea of partitioning paths starting from $x_i$ in the proof of Lemma~\ref{lemma.trans2}, and using Lemmas \ref{lemma.trans1} and \ref{lemma.trans2}, we obtain
  \begin{align*}
   Q_{V^{*}}(x_i,w^n_2) & =Q_{V^{*}}(y_{i-1},w^n_2) + Q_{V^{*}}(y_{i-2},w^{n}_2) + Q_{V^{*}}(x_{i-1},w^{n}_2) & \text{(by \eqref{eq.Q*induc2})} \\
     & =  Q_{V^{*}}(x_{i-1},w^{n-1}_1) + Q_{V^{*}}(x_{i-2},w^{n-1}_1) + Q_{V^{*}}(x_{i-1},w^{n}_2) & \text{(by \eqref{eq.trans1})} \\
     & = \left(Q_{V^{*}}(x_{i},w^{n-1}_2) + (-1)^{i-1}\right) + \left(Q_{V^{*}}(x_{i-1},w^{n-1}_2)+(-1)^{i-2} \right) & \\
     & \quad + Q_{V^{*}}(x_{i-1},w^{n}_2) & \text{(by \eqref{eq.Q*-invDbar})} \\
     & = Q_{V^{*}}(x_i,w^{n-1}_2) + Q_{V^{*}}(x_{i-1},w^{n-1}_2) + Q_{V^{*}}(x_{i-1},w^n_2), &
  \end{align*}
   for $3 \leq i \leq n-1$, as desired. This completes the proof of Lemma \ref{lemma.rec}.
\end{proof}

Now, we are ready to prove the third property in Proposition \ref{prop.Oreverse} which states that $r^{n}_{1,j} = r^{n}_{1,j-1} + r^{n-1}_{1,j} + r^{n-1}_{1,j-1}$, for $n \geq 3$ and $2 \leq j \leq n-1$.
\begin{proof}[Proof of Proposition \ref{prop.Oreverse} (3)]
On the graph $\mathcal{D}(n)$, it suffices to show that
\begin{equation*}
    Q_{V^{*}}(u_j,w^n_1) = Q_{V^{*}}(u_{j-1},w^n_1) + Q_{V^{*}}(u_{j},w^{n-1}_1) + Q_{V^{*}}(u_{j-1},w^{n-1}_1),
\end{equation*}
for $n \geq 3$ and $2 \leq j \leq n-1$.

When $j=2$ and $n \geq 3$, this can be computed directly. On the left-hand side, since the first path must connect $u_2$ with $v_3$, the second path must connect $x_1^n$ with $v_4$, and there are $2n-4$ such paths that do not intersect with the first path. As a consequence,
\begin{equation}\label{eq.r12}
  Q_{V^{*}}(u_2,w^n_1) = 2n-4.
\end{equation}
On the right-hand side, we have
\begin{equation*}
  Q_{V^{*}}(u_1,w^n_1) + Q_{V^{*}}(u_2,w^{n-1}_1) + Q_{V^{*}}(u_1,w^{n-1}_1) = 1+ \left(2(n-1)-4\right)+1 = 2n-4.
\end{equation*}

When $3 \leq j \leq n-1$ and $n \geq 4$, we again partition the paths starting from $u_j$ into the disjoint union of two sets, and then use Lemmas \ref{lemma.trans2} and \ref{lemma.rec}. As a result,
\begin{align*}
  Q_{V^{*}}(u_j,w^n_1) & = Q_{V^{*}}(u_{j+1},w^n_2) & \text{(by \eqref{eq.Q*-invDbar})}\\
   & = Q_{V^{*}}(x_{j+1},w^n_2) + Q_{V^{*}}(x_{j},w^n_2) & \text{(by \eqref{eq.Q*Dpartition})}\\
   & = Q_{V^{*}}(x_{j},w^n_2) + Q_{V^{*}}(x_{j+1},w^{n-1}_2) + Q_{V^{*}}(x_{j},w^{n-1}_2) & \\
   & \quad + Q_{V^{*}}(x_{j-1},w^n_2) + Q_{V^{*}}(x_{j},w^{n-1}_2) + Q_{V^{*}}(x_{j-1},w^{n-1}_2) & \text{(by \eqref{eq.lemmarec})}\\
   & = Q_{V^{*}}(u_{j},w^n_2) + Q_{V^{*}}(u_{j+1},w^{n-1}_2) + Q_{V^{*}}(u_{j},w^{n-1}_2) & \text{(by \eqref{eq.Q*Dpartition})} \\
   & = Q_{V^{*}}(u_{j-1},w^n_1) + Q_{V^{*}}(u_{j},w^{n-1}_1) + Q_{V^{*}}(u_{j-1},w^{n-1}_1), & \text{(by \eqref{eq.Q*-invDbar})}
\end{align*}
as desired. This completes the proof of Proposition \ref{prop.Oreverse} (3).
\end{proof}

We close this subsection with a useful identity of the number of pairs of non-intersecting paths on $\mathcal{D}(n)$, which will be applied to prove the last part of Proposition \ref{prop.Oreverse}.
\begin{lemma}\label{lemma.w1}
  On the graph $\mathcal{D}(n)$, for $n \geq 2$, we have
  \begin{equation}\label{eq.w1}
    Q_{V^{*}}(u_{n},w^n_1) = \frac{1}{2} Q_{V^{*}}(u_{n-1},u_n) +(-1)^{n-1}.
  \end{equation}
\end{lemma}
\begin{proof}
  Note that if two points $a$ and $b$ coincide on $\mathcal{D}(n)$, then $Q_{V^{*}}(a,b) = 0$ since the pair of paths start at the same point, they are always intersecting. We partition the path starting from $u_n$ and $u_{n-1}$ into the disjoint union of two sets, respectively. On the left-hand side of \eqref{eq.w1}, we obtain
  \begin{equation}\label{eq.w1-1}
    Q_{V^{*}}(u_{n},w^n_1) = Q_{V^{*}}(x_{n},w^n_1) + Q_{V^{*}}(x_{n-1},w^n_1) = Q_{V^{*}}(x_{n-1},w^n_1),
  \end{equation}
  because the points $x_n$ and $w^n_1$ coincide. On the right-hand side of \eqref{eq.w1}, we have
  \begin{align}
   & \frac{1}{2} Q_{V^{*}}(u_{n-1},u_n) +(-1)^{n-1} \nonumber \\
   & \quad = \frac{1}{2} \left( Q_{V^{*}}(x_{n-2},x_{n-1}) + Q_{V^{*}}(x_{n-2},x_{n}) + Q_{V^{*}}(x_{n-1},x_{n-1}) +  Q_{V^{*}}(x_{n-1},x_{n}) \right) + (-1)^{n} \nonumber \\
   & \quad = \frac{1}{2} \left( Q_{V^{*}}(x_{n-2},w_1^{n-1}) + Q_{V^{*}}(x_{n-2},w_1^{n}) + Q_{V^{*}}(x_{n-1},w_1^{n}) \right) + (-1)^{n}. \label{eq.w1-2}
  \end{align}

Applying Lemma~\ref{lemma.trans2} to the first two terms in \eqref{eq.w1-2},  we obtain
    \begin{align*}
      & \frac{1}{2} \left( Q_{V^{*}}(x_{n-1},w_2^{n-1}) + (-1)^{n-2} + Q_{V^{*}}(x_{n-1},w_2^{n}) + (-1)^{n-2} + Q_{V^{*}}(x_{n-1},w_1^{n}) \right) + (-1)^{n}  \\
      & \quad = \frac{1}{2} \left( Q_{V^{*}}(x_{n-1},w_2^{n-1}) + Q_{V^{*}}(x_{n-1},w_2^{n}) + Q_{V^{*}}(x_{n-1},w_1^{n}) \right)  \\
      & \quad = Q_{V^{*}}(x_{n-1},w_1^{n}),
    \end{align*}
because the last equality follows from partitioning the paths starting from $w^n_1$, $Q_{V^{*}}(x_{n-1},w^n_1) = Q_{V^{*}}(x_{n-1},w_2^{n-1}) + Q_{V^{*}}(x_{n-1},w_2^{n})$. This equals to the right-hand side of \eqref{eq.w1-1}, and thus completes the proof of Lemma~\ref{lemma.w1}.
\end{proof}

\subsection{Involutory matrix}\label{sec.involutory}

In this subsection, we will prove the last property of Proposition~\ref{prop.Oreverse} which states that the inverse of the matrix $R_n$ is itself. Equivalently, we want to show that the square of $R_n$ is the identity matrix of order $n$. Recall that the matrix $R_n$ is upper triangular, $r_{i,j}^n = 0$ if $i>j$. By convention, we set $r_{i,j}^n = 0$ if $i,j \notin [n]$.

The absolute value of the entries in the first row of $R_n$ form a triangular array $[r_{1,j}^n]_{1 \leq j \leq n}$; see Table~\ref{tab.r1jn} for $n=7$. The reader can verify the three-term recurrence relation (Proposition~\ref{prop.Oreverse} (3)) for the first few terms from this table. Now, we introduce a new square array $T=[t_{i,j}]_{i,j \geq 1}$ as follows.
\begin{definition}\label{def.array}
  Let $T=[t_{i,j}]_{i,j \geq 1}$ be a square array with entries satisfying a three-term recurrence relation $t_{i,j} = t_{i-1,j-1} + t_{i-1,j} + t_{i,j-1}$ for $i,j \geq 2$, and with initial values $t_{i,1} = 1$ and $t_{1,j} = g_j$ for $i,j \geq 1$, where the generating function of the sequence $(g_j)_{j \geq 1}$ is given by
\begin{equation}\label{eq.gfseq}
  G(x) = \sum_{j \geq 1}g_j x^{j-1} = \frac{(1+x)(1-2x-x^2)}{(1-x)(1+2x-x^2)}.
\end{equation}
\end{definition}
By convention, we set $t_{i,j} = 0$ if $i<1$ or $j<1$. Table~\ref{tab.square} displays the square array $T$ for $1 \leq i,j \leq 7$. The reader might notice a surprising fact that the lower triangular part of these two arrays is equal in Tables~\ref{tab.r1jn} and \ref{tab.square}, we will show in the following lemma that this holds in general.
\begin{table}[!htb]
    \begin{minipage}{0.5\linewidth}
      \centering
        \begin{tabular}{c|cccccccc}
          \diagbox{$n$}{$j$} & 1 & 2 & 3 & 4 & 5 & 6 & 7 \\
          \hline
          1 & 1 &  &  &  &  &  &  \\
          2 & 1 & 0 &  &  &  &  &  \\
          3 & 1 & 2 & 2 &  &  &  &  \\
          4 & 1 & 4 & 8 & 4 &  &  &  \\
          5 & 1 & 6 & 18 & 30 & 18 &  &  \\
          6 & 1 & 8 & 32 & 80 & 128 & 72 &  \\
          7 & 1 & 10 & 50 & 162 & 370 & 570 & 322 \\
        \end{tabular}
        \vspace{3mm}
        \caption{The triangular array $[r_{1,j}^n]_{1\leq j \leq n}$ up to $n=7$.}\label{tab.r1jn}
    \end{minipage}%
    \begin{minipage}{0.5\linewidth}
      \centering
        \begin{tabular}{c|cccccccc}
          \diagbox{$i$}{$j$} & 1 & 2 & 3 & 4 & 5 & 6 & 7 \\
          \hline
          1 & 1 & -2 & 2 & -10 & 18 & -50 & 114 \\
          2 & 1 & 0 & 0 & -8 & 0 & -32 & 32 \\
          3 & 1 & 2 & 2 & -6 & -14 & -46 & -46 \\
          4 & 1 & 4 & 8 & 4 & -16 & -76 & -168 \\
          5 & 1 & 6 & 18 & 30 & 18 & -74 & -318 \\
          6 & 1 & 8 & 32 & 80 & 128 & 72 &  -320\\
          7 & 1 & 10 & 50 & 162 & 370 & 570 & 322 \\
        \end{tabular}
        \vspace{3mm}
        \caption{The square array $T=[t_{i,j}]_{1 \leq i,j \leq 7}$.}\label{tab.square}
    \end{minipage}
\end{table}

\begin{lemma}\label{lemma.coincide}
  The lower triangular part of the square array $[t_{n,j}]_{1 \leq j \leq n}$ defined in Definition \ref{def.array} coincides with the triangular array $[r_{1,j}^n]_{1 \leq j \leq n}$, for all $n$.
\end{lemma}
\begin{proof}
  Since these two arrays satisfy the same three-term recurrence relation and their first columns are equal ($r^n_{1,1}=1$ by Proposition~\ref{prop.Oreverse} (1)), it suffices to show that their diagonal entries coincide. We will compute the generating function of these diagonal entries.

First, let $T(x,y) = \sum_{i \geq 1} \sum_{j \geq 1}t_{i,j}x^{i-1}y^{j-1}$ denote the bivariate generating function of the square array $T$. It is not hard to see that
\begin{align*}
   T(x,y) & = \sum_{i \geq 2} \sum_{j \geq 2} t_{i,j}x^{i-1}y^{j-1} + \sum_{i \geq 1} t_{i,1}x^{i-1} + \sum_{j \geq 1} t_{1,j} y^{j-1} - t_{1,1} \\
          & = \sum_{i \geq 2} \sum_{j \geq 2} \left( t_{i-1,j} + t_{i,j-1} + t_{i-1,j-1} \right)x^{i-1}y^{i-1} + \frac{1}{1-x} + G(y) -1 \\ 
          & = (x+y+xy) T(x,y) - \frac{x}{1-x} - yG(y) + \frac{1}{1-x} + G(y) -1,
\end{align*}
where $G(y)$ is the generating function given in \eqref{eq.gfseq}. As a consequence, we have
\begin{equation}\label{eq.gfT}
  T(x,y) = \frac{(1-y)G(y)}{1-x-y-xy}.
\end{equation}

The method of extracting the diagonal terms from a multivariate generating function can be found for instance in \cite[Chapter 6.3]{ECII}. Viewing $T(s,z/s) =  \sum_{i \geq 1} \sum_{j \geq 1}t_{i,j}s^{i-j}z^{j-1}$ as a Laurent series in $s$, then the generating function of the diagonal terms of $T$ is given by the constant term of $T(s,z/s)$. By Cauchy's integral formula and the residue theorem, we have
\begin{equation}\label{eq.gfdiagLaurent}
  \tilde{T}(z) = \sum_{n \geq 1} t_{n,n}z^{n-1} = [s^0] T \left( s,z/s \right) = \frac{1}{2 \pi i} \int_{|s|=r} T\left(s,z/s\right) \frac{ds}{s} = \sum_{s_0 = s(z)} \mathsf{Res}_{s_0} \frac{1}{s} T \left( s,z/s \right),
\end{equation}
where the sum ranges over all singularities $s(z)$ of $\frac{1}{s} T(s,z/s)$ inside a circle $|s|=r$, and they satisfy $\lim_{z \rightarrow 0} s(z)=0$. In Table \ref{tab.residue}, we list the singularities mentioned above and calculate their residues.
\begin{table}[htb!]
\centering
  \begin{tabular}{c|c|c|c|c}
    $i$       & $1$ & $2$ & $3$ & $4$ \\
    \hline & & & & \\[-4mm]
    Singularities $s_i$ & $0$ & $-z-\sqrt{2}z$ & $-z+\sqrt{2}z$ & $\displaystyle \frac{1}{2}\left( 1-z-\sqrt{1-6z+z^2} \right)$ \\
    \hline & & & & \\[-4mm]
    $\displaystyle \mathsf{Res}_{s_i} \frac{1}{s} T \left( s,z/s \right)$ & $-1$ & $\displaystyle \frac{2-\sqrt{2}}{1+z}$ & $\displaystyle \frac{2+\sqrt{2}}{1+z}$  & $\displaystyle \frac{-3+z-\sqrt{1-6z+z^2}}{2(1+z)}$  \\
  \end{tabular}
  \vspace{0.3cm}
  \caption{The singularities of $\frac{1}{s} T(s,z/s)$ satisfying $\lim_{z \rightarrow 0} s(z)=0$ and their residues.}\label{tab.residue}
\end{table}

Therefore, we have the generating function of the diagonal terms of the square array $T$.
\begin{equation}\label{eq.gfdiagT}
  \tilde{T}(z) = \sum_{i=1}^{4} \mathsf{Res}_{s_i} \frac{1}{s} T \left( s,z/s \right) = \frac{3-z-\sqrt{1-6z+z^2}}{2(1+z)}.
\end{equation}

Second, by Lemma~\ref{lemma.w1}, we know that $r_{1,n}^n = \frac{1}{2}Q_{V^{*}}(u_{n-1},u_n) + (-1)^{n-1}$, for $n \geq 2$. The following explicit expression of $Q_{V^{*}}(u_{n-1},u_n)$ is presented in the previous work \cite[Corollary~5]{Lee23} (take $i=n-1$ and $j=n$).
\begin{equation}\label{eq.expdiagr}
  \frac{1}{2}Q_{V^{*}}(u_{n-1},u_n) = \sum_{\ell = 1}^{n-1}(-1)^{\ell-1}\mathsf{Sch}_{n-1-\ell},
\end{equation}
where $\mathsf{Sch}_{p}$ is the $p$th large Schr{\"o}der number \cite[A006318]{OEIS}. Recall that the generating function of the large Schr{\"o}der numbers is $\mathsf{Sch}(z) = \sum_{n \geq 0} \mathsf{Sch}_n z^n = \left( 1-z-\sqrt{1-6z+z^2} \right)/(2z)$.

From the above discussion, we obtain the generating function of the entries $(r_{1,n}^{n})_{n \geq 1}$.
\begin{align}
  \sum_{n \geq 1}r_{1,n}^{n} z^{n-1} &= 1+ \sum_{n \geq 2}\frac{1}{2}Q_{V^{*}}(u_{n-1},u_n)z^{n-1} + \sum_{n \geq 2}(-1)^{n-1}z^{n-1} \nonumber \\
        & = 1+ \sum_{n \geq 2} \left( \sum_{\ell = 1}^{n-1}(-1)^{\ell-1}\mathsf{Sch}_{n-1-\ell} \right)z^{n-1} + \sum_{n \geq 2} (-1)^{n-1}z^{n-1} \nonumber \\
   & = 1 + \mathsf{Sch}(z) \frac{1}{1+z} + \frac{-z}{1+z}  = \frac{3-z-\sqrt{1-6z+z^2}}{2(1+z)}, \label{eq.gfdiagR}
\end{align}
where the third equality follows from the convolution of two sequences $(\mathsf{Sch}_n)_{n \geq 1}$ and $\left( (-1)^{n} \right)_{n\geq 1}$. Two generating functions \eqref{eq.gfdiagT} and \eqref{eq.gfdiagR} coincide. This completes the proof of Lemma~\ref{lemma.coincide}.
\end{proof}

We are now ready to prove the last property of Proposition~\ref{prop.Oreverse} which states that the inverse of the matrix $R_n$ is itself.

\begin{proof}[Proof of Proposition \ref{prop.Oreverse} (4)]
  By the direct matrix multiplication, the $(i,j)$-entry of $(R_n)^2$ is given by
    \begin{equation}\label{eq.cij}
      c_{i,j} = \sum_{k=1}^{n} (-1)^{k+j} r^n_{i,k} r^n_{k,j}.
    \end{equation}
  Thanks to the translation invariant property of $r_{i,j}^n$ (Proposition \ref{prop.Oreverse} (2)) and the convention stated in the first paragraph of this subsection, we may rewrite \eqref{eq.cij} as follows.
    \begin{equation*}
      c_{i,j}  = \sum_{k=1}^{n} (-1)^{k+j} r^n_{i+1,k+1} r^n_{k+1,j+1} = \sum_{k=2}^{n+1} (-1)^{k-1+j} r^n_{i+1,k} r^n_{k,j+1} = \sum_{k=1}^{n} (-1)^{k+1+j} r^n_{i+1,k} r^n_{k,j+1} = c_{i+1,j+1}.
    \end{equation*}
  Therefore, instead of showing $c_{i,j} = \delta_{i,j}$, it is enough to show that
    \begin{equation}\label{eq.cijreduce}
      c_{1,j} = \sum_{k=1}^{n} (-1)^{k+j} r^n_{1,k} r^n_{k,j} = \sum_{k=1}^{j} (-1)^{k+j} r^n_{1,k} r^n_{1,j-(k-1)} = \delta_{1,j}, \text{ for $j=1,\dotsc,n$.}
    \end{equation}
Note that we are able to reduce the number of terms in \eqref{eq.cijreduce} since $r_{k,j}^n=0$ when the index $k>j$.

  Given a positive integer $n$, let $T=[t_{i,j}]$ be the array defined in Definition~\ref{def.array}. We will prove a more general identity below:
    \begin{equation}\label{eq.invo1}
      \sum_{k=1}^{j} (-1)^{k+j} t_{n,k}t_{n,j+1-k} = \delta_{1,j}, \text{ for all $j \geq 1$.}
    \end{equation}
Due to Lemma~\ref{lemma.coincide}, $t_{n,j} = r_{1,j}^n$ when $j=1,\dotsc,n$. Then \eqref{eq.cijreduce} follows immediately from \eqref{eq.invo1}. The proof will be complete if we show that \eqref{eq.invo1} holds. Equivalently, we will show the following identity holds.
\begin{equation}\label{eq.invo2}
  \sum_{j \geq 1} \left( \sum_{k \geq1} (-1)^{k-1} t_{n,k}t_{n,j+1-k} \right) y^{j-1} = 1.
\end{equation}

Let $T_n(y) = \sum_{j \geq 1} t_{n,j}y^{j-1}$ be the generating function of the entries in the $n$th row of $T$. Notice that $T_n(y)$ is the coefficient of $x^{n-1}$ of $T(x,y)$, a simple calculation from \eqref{eq.gfT} gives
\begin{equation}\label{eq.gfTn}
  T_n(y) = \left( \frac{1+y}{1-y}\right)^{n-1}G(y).
\end{equation}

To prove \eqref{eq.invo2}, we interchange the sum in \eqref{eq.invo2}, simplify the sum by observing that it is the convolution of two sequences $(t_{n,k})_{k \geq 1}$ and $((-1)^{k-1}t_{n,k})_{k \geq 1}$. Using \eqref{eq.gfTn}, we finally obtain
\begin{align*}
  \sum_{j \geq 1} \sum_{k \geq1} (-1)^{k-1} t_{n,k}t_{n,j+1-k} y^{j-1} & = \sum_{k \geq 1} \sum_{j \geq1} (-1)^{k-1} t_{n,k}t_{n,j+1-k} y^{j-1} = \sum_{k \geq 1} \left( (-1)^{k-1} t_{n,k} \sum_{j \geq1} t_{n,j+1-k} y^{j-1} \right)\\
   & = \sum_{k \geq 1} \left( (-1)^{k-1} t_{n,k} y^{k-1} \sum_{j \geq1} t_{n,j+1-k} y^{j-k} \right) = T_n(-y) T_n(y) = 1.
\end{align*}
This completes the proof of Proposition \ref{prop.Oreverse} (4).
\end{proof}

\subsection{The symmetry property}\label{sec.symmetry}

We are now prepared to prove our first main theorem (Theorem~\ref{thm.main1}) which states the symmetry property for off-diagonally symmetric domino tilings of an odd-order Aztec diamond: $|O(2n-1;k)| = |O(2n-1;2n-k)|$, for $k=1,2,\dotsc,n$.

\begin{proof}[Proof of Theorem \ref{thm.main1}]
Given a positive integer $n$, let $v_k = |O(2n-1,k)| - |O(2n-1,2n-k)|$ for $k=1,\dotsc,2n-1$ (clearly, $v_k = -v_{2n-k}$). Define two column vectors
\begin{equation*}
  \mathbf{v}=(v_1,\dotsc,v_{n-1})^{\intercal} \text{ and } \mathbf{\overline{v}}=(v_1,\dotsc,v_{n-1},0,-v_{n-1},\dotsc,-v_1)^{\intercal}.
\end{equation*}
It suffices to show that $\mathbf{v}$ is the zero vector.

Let $J_{2n-1}$ be the matrix of order ${2n-1}$ whose anti-diagonal entries are $1$, and other entries are $0$. Note that matrices $R_{2n-1}$ (given in Proposition~\ref{prop.Oreverse}) and $J_{2n-1}$ both reverse the order of elements of the column vector $\mathbf{O_{2n-1}}$. Since the matrix $R_{2n-1}$ is involutory (Proposition~\ref{prop.Oreverse} (4)), we have $R_{2n-1}\mathbf{O}^{\mathsf{rev}}_{2n-1} = \mathbf{O}_{2n-1}$, and therefore
\begin{equation}\label{eq.pfsym1}
  (R_{2n-1}J_{2n-1} - R_{2n-1})\mathbf{O}_{2n-1} = \mathbf{O}_{2n-1} - \mathbf{O}^{\mathsf{rev}}_{2n-1} = \mathbf{\overline{v}}.
\end{equation}
Let $x_{i,j}^{2n-1}$ be the $(i,j)$-entry of the matrix $R_{2n-1}J_{2n-1} - R_{2n-1}$. It is easy to check that this matrix is symmetric (ignore the signs) about the vertical central line, that is, $x_{i,j}^{2n-1} = -x_{i,2n-j}^{2n-1}$, for $j=1,\dotsc,n-1$, and the $n$th column is zero.

Due to the symmetry mentioned above, \eqref{eq.pfsym1} can be rewritten as
\begin{equation}\label{eq.pfsym2}
  X_{2n-1} \mathbf{v} = \mathbf{\overline{v}},
\end{equation}
where $X_{2n-1}=[x_{i,j}^{2n-1}]_{1 \leq i \leq 2n-1,1 \leq j \leq n-1}$ is the $(2n-1)\times (n-1)$ matrix. It is equivalent to
\begin{equation}\label{eq.pfsym3}
  (X_{2n-1} - Y_{2n-1})\mathbf{v} = \mathbf{0},
\end{equation}
where $Y_{2n-1} = \begin{bmatrix}
                    I_{n-1} \\
                    0 \\
                    -J_{n-1}
                  \end{bmatrix}$ ($I_{n-1}$ is the identity matrix of order $n-1$).

Take $n=3$ as an example, the entries of $R_5$ can be found in the fifth row of Table \ref{tab.r1jn}.
\begin{equation*}
  R_{5}J_{5} - R_{5} = \begin{bmatrix}
                         18 & -30 & 18 & -6 & 1 \\
                         30 & -18 & 6  & -1 & 0 \\
                         18 & -6  & 1  & 0  & 0 \\
                         6  & -1  & 0  & 0  & 0 \\
                         1  &  0  & 0  & 0  & 0
                       \end{bmatrix} -
                        \begin{bmatrix}
                         1 & -6 & 18 & -30 & 18 \\
                         0 & -1 & 6  & -18 & 30 \\
                         0 & 0  & 1  & -6  & 18 \\
                         0 & 0  & 0  & -1  & 6 \\
                         0 & 0  & 0  & 0   & 1
                       \end{bmatrix} =
                       \begin{bmatrix}
                         17 & -24 & 0 & 24 & -17 \\
                         30 & -17 & 0 & 17 & -30 \\
                         18 & -6  & 0 & 6  & -18 \\
                         6  & -1  & 0 & 1  & -6 \\
                         1  & 0   & 0 & 0  & -1
                       \end{bmatrix},
\end{equation*}
and this can be rewritten as
\begin{equation*}
    (R_{5}J_{5} - R_{5})\mathbf{O}_5 =
  \begin{bmatrix}
     17 & -24 \\
     30 & -17 \\
     18 & -6  \\
     6  & -1  \\
     1  & 0
  \end{bmatrix}
\begin{bmatrix}
  v_1 \\
  v_2
\end{bmatrix} =
\begin{bmatrix}
   v_1  \\
   v_2  \\
     0  \\
  -v_2  \\
  -v_1
\end{bmatrix} \text{, \quad}
\left(
\begin{bmatrix}
     17 & -24 \\
     30 & -17 \\
     18 & -6  \\
     6  & -1  \\
     1  & 0
  \end{bmatrix} -
\begin{bmatrix}
     1 & 0 \\
     0 & 1 \\
     0 & 0 \\
     0 & -1 \\
    -1 & 0
  \end{bmatrix}
\right)
\begin{bmatrix}
  v_1 \\
  v_2
\end{bmatrix} =
\begin{bmatrix}
  0 \\
  0
\end{bmatrix}.
\end{equation*}

Our proof will be complete once we prove the following (which will imply that $\mathbf{v}$ is the zero vector).

\noindent \textbf{Claim.} The matrix $X_{2n-1}-Y_{2n-1}$ is full rank, that is, $\mathsf{rank}(X_{2n-1}-Y_{2n-1}) = n-1$.
\begin{proof}[Proof of Claim]
According to Proposition \ref{prop.Oreverse} (1) and the way we constructed the matrix $X_{2n-1}$, one can see that the $(2n-j,j)$-entry of $X_{2n-1} - Y_{2n-1}$ is $2$ if $j$ is odd and $0$ if $j$ is even; also, the $(i,j)$-entry of $X_{2n-1} - Y_{2n-1}$ is $0$ if $n+1 \leq i \leq 2n-1$ and $j>2n-i$. This implies that the non-zero entries of $X_{2n-1}-Y_{2n-1}$ form a staircase shape, we illustrate the first few columns below.
\begin{equation*}
    \begin{bmatrix}
  \vdots & \vdots & \vdots & \vdots & \vdots & \\
   * & * & * &  * & * & \cdots\\
   * &  * & * & -a_3 & a_2 & \cdots\\
   * & * & *  & -a_2 & a_1 & \cdots\\
   * & -a_3 & a_2 & -a_1 & 2 & \cdots\\
   a_3 & -a_2 & a_1 & 0 & 0 & \cdots\\
   a_2 & -a_1 & 2 & 0 & 0 & \cdots\\
   a_1 &  0  & 0 & 0 & 0 & \cdots\\
   2 & 0 & 0 & 0 & 0 & \cdots
\end{bmatrix},
\end{equation*}
where $*$ is a non-zero entry, $a_i=r^{2n-1}_{1,i+1}$ (first few entries are given in the $(2n-1)$th row in Table \ref{tab.r1jn}).

In the matrix $X_{2n-1}-Y_{2n-1}$, one can see from Proposition \ref{prop.Oreverse} (2) that the first three non-zero entries (counted from the bottom) of columns indexed by $2j$ and $2j+1$ are ``almost identical'', where $j=1,2,\dotsc,\lfloor \frac{n-2}{2} \rfloor$. We should point out that the only difference occurs when $n$ is even, the ``$a_2$'' entry of the last column is actually given by $a_2-2$. It suffices to show that two columns indexed by $2j$ and $2j+1$ are independent, and thus the matrix $X_{2n-1}-Y_{2n-1}$ has $n-1$ independent columns. 

It is not hard to solve the three-term recurrence relation of $r^{n}_{1,j}$ (Proposition \ref{prop.Oreverse} (3)) for $j=2,3$ and $4$, we then obtain the following formulas.
\begin{equation*}
    r^{n}_{1,2}=2(n-1), \quad r^{n}_{1,3}=2(n-1)^2 \text{\quad and \quad} r^{n}_{1,4}=\frac{2}{3}(2n^3-12n^2+25n-30).
\end{equation*}
Once a positive integer $n$ is given, it is clear from the formulas that 
\begin{equation*}
    \frac{-a_1}{2} = \frac{-2(2n-3)}{2} = -(2n-3) = \frac{-a_2}{a_1}.
\end{equation*}
However, one can check that $\frac{-a_1}{2} \neq \frac{-a_3}{a_2}$ and $\frac{-a_1}{2} \neq \frac{-a_3}{a_2-2}$ in general. As a result, two columns indexed by $2j$ and $2j+1$ of $X_{2n-1} - Y_{2n-1}$ are independent, for $j=1,2,\dotsc,\lfloor \frac{n-2}{2} \rfloor$. This completes the proof of Theorem \ref{thm.main1}.
\end{proof}
\phantom{\qedhere}
\end{proof}
\begin{remark}
  The reader might notice some alternative ways to prove Theorem~\ref{thm.main1}. One way is to apply Theorem~\ref{thm.OffPfaffian} by showing that the Pfaffian of $A_I$ when $I=[2n-1] \setminus \{k\}$ and $I=[2n-1] \setminus \{2n-k\}$ are equal. However, evaluating $\pf(A_I)$ directly seems challenging, and no progress has been made so far. Another approach is to find a bijection between domino tilings of these two sets $|O(2n-1;k)|$ and $|O(2n-1;2n-k)|$ directly. We leave these directions to be pursued by the interested reader.
\end{remark}

\section{Nearly off-diagonally symmetric domino tilings of the Aztec diamond}\label{sec.nearly}

In this section, we will study nearly off-diagonally symmetric domino tilings of the Aztec diamond of order $2n-1$ (with no boundary defect), where the $2n-2$ cells on the vertical diagonal are assigned $0$ while only one cell is assigned either $-1$ or $1$. In Section~\ref{sec.pfmain2}, we prove our second main result (Theorem~\ref{thm.main2}). The third main result (Theorem~\ref{thm.main3}) will be proved in Section~\ref{sec.pfmain3}. Interesting identities (Corollary \ref{cor.identity}) involving the numbers of off-diagonally symmetric domino tilings that we considered in this paper are shown in Section~\ref{sec.misid}.

\subsection{Proof of Theorem \ref{thm.main2}}\label{sec.pfmain2}

We remind the reader that $D^{-}(2n-1;k)$ (resp., $D^{+}(2n-1;k)$) is the set of nearly off-diagonally symmetric domino tilings of $AD(2n-1)$ in which the $k$th cell (from bottom to top) on the vertical diagonal is assigned $-1$ (resp., $1$). $D^{\pm}(2n-1;k) = D^{-}(2n-1;k) \cup D^{+}(2n-1;k)$, and $D(2n-1) = \bigcup_{k=1}^{2n-1} D^{\pm}(2n-1;k)$.

On the graph $\mathcal{\overline{D}}(n)$, it is easy to see that the number of paths going from $x_i$ to $v_{\ell}$ is given by the Delannoy number
\begin{equation}\label{eq.pxv}
  |\mathscr{P}(x_i,v_{\ell})| = \begin{cases}
                d_{i-j,j-1}, & \text{if $\ell=2j$,} \\
                d_{i-j,j-2}, & \text{if $\ell=2j-1$,}
            \end{cases}
\end{equation}
for $1 \leq i \leq n$ and $1 \leq \ell \leq 2n$.
On the graph $\mathcal{D}(n)$, using \eqref{eq.partitionu}, \eqref{eq.pxv} and the recurrence relation of Delannoy numbers, we obtain
\begin{align}
  |\mathscr{P}(u_i,v_{2j})| + |\mathscr{P}(u_i,v_{2j-1})| &= 2 d_{i-j,j-1},    \label{eq.puv+} \\
  |\mathscr{P}(u_i,v_{2j})| - |\mathscr{P}(u_i,v_{2j-1})| &= 2 d_{i-j-1,j-1},  \label{eq.puv-}
\end{align}
  for $1 \leq i,j \leq n$.

Now, we are ready to prove our second main result (Theorem~\ref{thm.main2}), which provides the Pfaffian expression of $D(2n-1)$.

\begin{proof}[Proof of Theorem \ref{thm.main2}]
  Consider the graph $\mathcal{D}(2n-1)$, by Lemma~\ref{lemma.cell}, we have 
  \begin{equation}\label{eq.main2proof-0}
      |D^{\pm}(2n-1;k)| = |\mathscr{P}_0(U,\{v_{2k}\} \cup V^{*})| + |\mathscr{P}_0(U,\{v_{2k-1}\} \cup V^{*})|,
  \end{equation}
  where $U=\{u_1,\dotsc,u_{2n-1}\}$ is the set of fixed starting points, $V^{*}=\{v_1^{*},\dotsc,v_{2n-1}^{*}\}$ is the set of doublets and $\overline{V} = \{\overline{v}\}$ ($\overline{v} = v_{2k}$ or $v_{2k-1}$) is the only fixed ending point of a path.

 To obtain the number of these non-intersecting paths, we apply Lemma \ref{lemma.mod2} with the graph $\mathcal{G} = \mathcal{D}(2n-1)$ and the sets of points described above ($|\overline{V}|=m=1$). We note that the sets $U$ and $\overline{V} \cup V^{*}$ are not compatible in our case. As a result, for $\overline{v} = v_{2k}$ or $v_{2k-1}$, we have
\begin{equation}\label{eq.main2proof-1}
   \sum_{\pi \in \mathfrak{S}_{2n-1}} \sgn(\pi) |\mathscr{P}_{0}^{\pi}(U,\{\overline{v}\} \cup V^{*})| = \pf
   \begin{bmatrix}
      Q^{*} & H(\overline{v}) \\
      -H(\overline{v})^{\intercal} & 0
      \end{bmatrix},
\end{equation}
where $Q^{*}$ is the matrix $A_{[2n-1]}$ given in \eqref{eq.rec} and $H(\overline{v})=[|\mathscr{P}(u_i,\overline{v})|]_{i=1}^{2n-1}$ is the column vector.

We remind the reader that the set $\mathscr{P}_0^{\pi}(U,\{\overline{v}\} \cup V^{*})$ consists of $(2n-1)$-tuples of non-intersecting paths  $(p_1,\dotsc,p_{2n-1})$ with the connection type $\pi$, where
\begin{itemize}
  \item $p_1$ connects $u_{\pi(1)}$ with the fixed ending point $\overline{v}$, and
  \item for $i=2,\dotsc,2n-1$, $p_i$ connects $u_{\pi(i)}$ with some doublet in $V^{*}$ (two elements in a doublet are both the ending points of paths or neither the ending points of paths).
\end{itemize}
Due to the condition that the paths $p_2,\dotsc,p_{2n-1}$ end at some doublets, it turns out that all the possible connection types of non-intersecting paths are of the form written in the following two-line notation:
\begin{equation}\label{eq.pinearlyoff}
  \pi = \begin{pmatrix}
          1 & 2 & 3  & \cdots & 2\ell-2 & 2\ell-1 & 2\ell & 2\ell+1 & \cdots & 2n-1 \\
          2\ell-1 & 1 & 2 & \cdots & 2\ell -3 & 2\ell-2 & 2\ell & 2\ell+1 & \cdots & 2n-1
        \end{pmatrix},
\end{equation}
for $\ell = 1,\dotsc,n$. It is easy to see that $\pi$ is always an even permutation, that is, $\sgn(\pi) = 1$.

Therefore, the left-hand side of \eqref{eq.main2proof-1} gives the correct total number of non-intersecting paths:
\begin{equation}\label{eq.main2proof-2}
  \sum_{\pi \in \mathfrak{S}_{2n-1}} \sgn(\pi) |\mathscr{P}_{0}^{\pi}(U,\{\overline{v}\} \cup V^{*})| = |\mathscr{P}_{0}(U,\{\overline{v}\} \cup V^{*})| = \pf
   \begin{bmatrix}
      Q^{*} & H(\overline{v}) \\
      -H(\overline{v})^{\intercal} & 0
      \end{bmatrix}.
\end{equation}
By \eqref{eq.main2proof-0}, \eqref{eq.main2proof-2} and the linearity of Pfaffians, we have
\begin{align}
    |D(2n-1)| & = \sum_{k=1}^{2n-1} |D^{\pm}(2n-1;k)|  = \sum_{k=1}^{4n-2} |\mathscr{P}_{0}(U,\{v_{k}\} \cup V^{*})| \nonumber \\
              & = \sum_{k=1}^{4n-2} \pf \begin{bmatrix}
                      Q^{*} & H(v_k)  \\
                      -H(v_k)^{\intercal} & 0
                  \end{bmatrix}
                = \pf \begin{bmatrix}
                      Q^{*} & H \\
                      -H^{\intercal} & 0
                  \end{bmatrix}, \label{eq.main2proof-3}
\end{align}
where the column vector $H = \sum_{k=1}^{4n-2} H(v_k) = [\sum_{k=1}^{4n-2} |\mathscr{P}(u_i,v_k)|]_{i=1}^{2n-1}$.

By \eqref{eq.puv+}, the $i$th component of $H$ can be expressed as
\begin{equation}\label{eq.main2proof-4}
  h_i = \sum_{k=1}^{4n-2} |\mathscr{P}(u_i,v_k)| = \sum_{j=1}^{2n-1} 2 d_{i-j,j-1}.
\end{equation}
In fact, $h_i$ is twice the sum of the $i$th anti-diagonal entries of the square array of Delannoy numbers $[d_{p,q}]_{p,q \geq 0}$. It is well-known that $(h_i)_{i \geq 1}$ is exactly twice the Pell numbers (see for example \cite[A008288]{OEIS}), which is the sequence $(f_i)_{i \geq 1}$ given in \eqref{eq.pell}. One can also check this fact by showing they satisfy the same recurrence relation and initial values, we leave this straightforward computation to the interested reader.

Therefore, $h_i = f_i$ for all $i$. The matrix in \eqref{eq.main2proof-3} coincides with the matrix $B_{2n}$ given in \eqref{eq.matrixB}. This completes the proof of Theorem~\ref{thm.main2}.
\end{proof}

\subsection{Proof of Theorem~\ref{thm.main3}}\label{sec.pfmain3}

We now combine the ideas of proving Proposition~\ref{prop.Oreverse} and Theorem~\ref{thm.main2}, the proof of our third main result (Theorem \ref{thm.main3}) is presented below.

\begin{proof}[Proof of Theorem~\ref{thm.main3}]
  By Lemma~\ref{lemma.cell} and the proof of Theorem~\ref{thm.main2}, we have the following expression. For $k=1,\dotsc,2n-1$,
\begin{align}
  |D^{\pm}(2n-1;k)| & = |\mathscr{P}_0(U,\{v_{2k}\} \cup V^{*})| + |\mathscr{P}_0(U,\{v_{2k-1}\} \cup V^{*})| \nonumber \\
                    &    = \pf \begin{bmatrix}
                        Q^{*} & H(v_{2k})+H(v_{2k-1})  \\
                        -(H(v_{2k})+H(v_{2k-1}))^{\intercal} & 0
                        \end{bmatrix}, \label{eq.main3proof-1} \\
  |D^{-}(2n-1;k)| & = |\mathscr{P}_0(U,\{v_{2k}\} \cup V^{*})| - |\mathscr{P}_0(U,\{v_{2k-1}\} \cup V^{*})|  \nonumber \\
                  &      = \pf \begin{bmatrix}
                            Q^{*} & H(v_{2k})-H(v_{2k-1})  \\
                              -(H(v_{2k})-H(v_{2k-1}))^{\intercal} & 0
                          \end{bmatrix} \label{eq.main3proof-2},
\end{align}
where $Q^{*}$ is the matrix $A_{[2n-1]}$ given in \eqref{eq.rec} and $H(v_j)=[|\mathscr{P}(u_i,v_j)|]_{i=1}^{2n-1}$ is the column vector.

Thanks to \eqref{eq.puv+} and \eqref{eq.puv-}, the column vectors in \eqref{eq.main3proof-1} and \eqref{eq.main3proof-2} are given by $H(v_{2k})+H(v_{2k-1}) = [2d_{i-k,k-1}]_{1 \leq i \leq 2n-1}$ and $H(v_{2k})-H(v_{2k-1}) = [2d_{i-k-1,k-1}]_{1 \leq i \leq 2n-1}$, respectively. Using a similar idea that led to \eqref{eq.R3}, we expand the Pfaffian in \eqref{eq.main3proof-1} and \eqref{eq.main3proof-2} along the last column. As a consequence, we obtain
\begin{align}
  |D^{\pm}(2n-1;k)| & = \sum_{\ell=1}^{2n-1} (-1)^{2n+\ell-1} 2d_{\ell-k,k-1} |O(2n-1;\ell)|, \\
  |D^{-}(2n-1;k)| & = \sum_{\ell=1}^{2n-1} (-1)^{2n+\ell-1} 2d_{\ell-k-1,k-1} |O(2n-1;\ell)|,
\end{align}
for $k=1,\dotsc,2n-1$. 

Clearly, these $2n-1$ equations are equivalent to the following matrix equations
\begin{equation*}
  M^{\pm}_{[2n-1]} \mathbf{O}_{2n-1} = \mathbf{D}^{\pm}_{2n-1}, \text{ and } M^{-}_{[2n-1]} \mathbf{O}_{2n-1} = \mathbf{D}^{-}_{2n-1},
\end{equation*}
where the $(k,\ell)$-entry of $M^{\pm}_{[2n-1]}$ is given by $2 (-1)^{\ell-1} d_{\ell-k,k-1}$ and the $(k,\ell)$-entry of $M^{-}_{[2n-1]}$ is given by $2 (-1)^{\ell-1} d_{\ell-k-1,k-1}$.

Finally, we have
$\mathbf{D}^{+}_{2n-1} = \mathbf{D}^{\pm}_{2n-1}-\mathbf{D}^{-}_{2n-1} = \left( M^{\pm}_{[2n-1]} -M^{-}_{[2n-1]} \right) \mathbf{O}_{2n-1} = M^{+}_{[2n-1]} \mathbf{O}_{2n-1}$. This completes the proof of Theorem~\ref{thm.main3}.
\end{proof}

\subsection{Proof of Corollary \ref{cor.identity}}\label{sec.misid}

We close this section by proving Corollary \ref{cor.identity}, this result follows from Theorem~\ref{thm.main1}, Theorem~\ref{thm.main3} and Proposition~\ref{prop.Oreverse}.
\begin{proof}[Proof of Corollary \ref{cor.identity}]
  By the symmetry property of $|O(2n-1;k)|$ (Theorem~\ref{thm.main1}) and the obvious fact that $|D^{\pm}(2n-1;k)| = |D^{\pm}(2n-1;2n-k)|$ (the Aztec diamond region is symmetric about the horizontal diagonal), \eqref{eq.cor1} is equivalent to
\begin{equation}\label{eq.cor3}
  \frac{|O(2n-1;2n-2)|}{|O(2n-1;2n-1)|} = 2n-3 \text{\quad and \quad} \frac{|D^{\pm}(2n-1;2n-2)|}{|D^{\pm}(2n-1;2n-1)|} = 2n-2.
\end{equation}
The first identity is obtained from the second to the last row of the matrix equation in \eqref{eq.RO}, the translation invariant of $r_{i,j}^{n}$ (Proposition \ref{prop.Oreverse} (2)) and \eqref{eq.r12} (replace $n$ by $2n-1$).
\begin{align*}
    |O(2n-1;2n-2)| & = - |O(2n-1;2n-2)| + \left( r^{2n-1}_{2n-2,2n-1} \right) |O(2n-1;2n-1)|  \\
         & = - |O(2n-1;2n-2)| + \left( r^{2n-1}_{1,2} \right) |O(2n-1;2n-1)|  \\
         & = - |O(2n-1;2n-2)| + (4n-6) |O(2n-1;2n-1)|.
\end{align*}

By Theorem \ref{thm.main3}, the last row of the matrix equation \eqref{eq.matrixod} gives
\begin{equation}\label{eq.cor4}
  |D^{\pm}(2n-1;2n-1)| = \left( m^{\pm}_{2n-1,2n-1} \right) |O(2n-1;2n-1)| = 2 |O(2n-1;2n-1)|.
\end{equation}
The second to the last row of the matrix equation \eqref{eq.matrixod} leads to
\begin{align*}
    |D^{\pm}(2n-1;2n-2)| & = -2 |O(2n-1;2n-2)| + \left( m^{\pm}_{2n-2,2n-1} \right) |O(2n-1;2n-1)|  \\
         & = -2 |O(2n-1;2n-2)| + 2(d_{1,2n-3}) |O(2n-1;2n-1)|  \\
         & = -2(2n-3) |O(2n-1;2n-1)| + 2(4n-5) |O(2n-1;2n-1)| \\
         & = (4n-4) |O(2n-1;2n-1)|  = (2n-2) |D^{\pm}(2n-1;2n-1)|.
\end{align*}
This finishes the proof of \eqref{eq.cor1}.

Finally, the entries in the first row of the matrix $M^{\pm}$ are given by $m^{\pm}_{1,j} = 2(-1)^{j-1}d_{j-1,0} = 2(-1)^{j-1}$. Then the first row of the matrix equation \eqref{eq.matrixod} gives the following alternating sum
\begin{equation*}
  \sum_{j=1}^{2n-1} 2 (-1)^{j-1} |O(2n-1;j)| = |D^{\pm}(2n-1;1)| = |D^{\pm}(2n-1;2n-1)| = 2 |O(2n-1;2n-1)|,
\end{equation*}
which is equivalent to \eqref{eq.cor2}. This completes the proof of Corollary \ref{cor.identity}.
\end{proof}

\section{Open problems}\label{sec.open}

In this section, we formulate two conjectures about the number of domino tilings of $O(2n-1;k)$, $D^{*}(2n-1;k)$ ($*$ represents $\pm,+$ and $-$, respectively) and $D(2n-1)$ studied in this paper. Both conjectures are verified up to $n=35$.

A sequence of real numbers $(a_i)_{i=0}^{n}$ is said to be \textit{unimodal} if there is an index $j$ such that $a_0 \leq a_1 \leq \cdots a_{j-1} \leq a_j \geq a_{j+1} \geq \cdots \geq a_{n}$. This sequence is \textit{log-concave} if $a_{i-1}a_{i+1} \leq a_i^2$ for $i=1,\dotsc,n-2$. Recall that a positive and log-concave sequence implies this sequence is unimodal.

Given a positive integer $n$, our data shows that the sequences $(|D^{*}(2n-1;k)|)_{k=1}^{2n-1}$ ($*$ represents $\pm,+$ and $-$, respectively) are not unimodal, but we have the following conjecture. 
\begin{conjecture}\label{conj.logconcave}
  Given a positive integer $n$, the sequence $(|O(2n-1;k)|)_{k=1}^{2n-1}$ consisting of the numbers of off-diagonally symmetric domino tilings of $AD(2n-1)$ with one boundary defect at the position $k$ is log-concave.
\end{conjecture}

We know that the number of domino tilings of $AD(n)$ is given by the simple formula $\M(AD(n)) = 2^{n(n+1)/2}$. Consequently, we have $\lim_{n \rightarrow \infty} (\M(AD(n)))^{1/n^2} = \sqrt{2}$. Our data shows that the number of (nearly) off-diagonally symmetric domino tilings of the Aztec diamond has a similar asymptotic behavior, which is stated below.
\begin{conjecture}\label{conj.asymptotic}
  Let $O(2n,[2n])$ be the set of off-diagonally symmetric domino tilings of $AD(2n)$ with no boundary defect and $D(2n-1)$ be the set of nearly off-diagonally symmetric domino tilings of $AD(2n-1)$. Then we have
\begin{equation}\label{eq.asymptotic}
  \lim_{n \rightarrow \infty} \left( |O(2n,[2n])| \right)^{\frac{2}{(2n)^2}} = \lim_{n \rightarrow \infty} \left( |D(2n-1)| \right)^{\frac{2}{(2n-1)^2}} = \sqrt{2}.
\end{equation}
\end{conjecture}
%

\noindent \textbf{Acknowledgements.} The author thanks Mihai Ciucu for stimulating discussions and helpful suggestions throughout the writing of this paper. The author also thanks the reviewers for carefully reading the manuscript and providing helpful comments.




\end{document}